\documentclass{birkjour}
\usepackage{eucal,enumerate,mathrsfs,xspace}
\usepackage{amsmath,amssymb,epsfig,bbm}
\usepackage[usenames]{color}

\numberwithin{equation}{section}

\newtheorem{theorem}{Theorem}[section]

\newtheorem{lemma}[theorem]{Lemma}
\newtheorem{proposition}[theorem]{Proposition}
\theoremstyle{definition}
\newtheorem{definition}[theorem]{Definition}

\newcommand{\N}{\mathbb{N}}

\newcommand{\R}{\mathbb{R}}

\renewcommand{\AA}{\mathscr{A}}

\newcommand{\KK}{\mathscr{K}}

\newcommand{\QQ}{\mathscr{Q}}

\newcommand{\XX}{\mathscr{X}}
\newcommand{\cB}{{\ensuremath{\mathcal B}}}

\newcommand{\cE}{{\ensuremath{\mathcal E}}}
\newcommand{\cF}{{\ensuremath{\mathcal F}}}

\newcommand{\cI}{{\ensuremath{\mathcal I}}}
\newcommand{\cL}{{\ensuremath{\mathcal L}}}

\newcommand{\cP}{{\ensuremath{\mathcal P}}}
\newcommand{\cR}{{\ensuremath{\mathcal R}}}
\newcommand{\cT}{{\ensuremath{\mathcal T}}}

\renewcommand{\tt}{{\mbox{\boldmath$t$}}}

\newcommand{\sfa}{{\sf a}}
\newcommand{\sfb}{{\sf b}}
\newcommand{\sfc}{{\sf c}}
\newcommand{\sfd}{{\sf d}}

\newcommand{\sfp}{{\sf p}}
\newcommand{\sfq}{{\sf q}}
\newcommand{\sfr}{{\sf r}}
\newcommand{\sfs}{{\sf s}}
\newcommand{\sft}{{\sf t}}
\newcommand{\sfu}{{\sf u}}

\newcommand{\sfw}{{\sf w}}

\newcommand{\sF}{{\sf F}}

\newcommand{\sfL}{{\sf L}}

\newcommand{\sfS}{{\sf S}}

\newcommand{\fre}{{\frak e}}
\newcommand{\frf}{{\frak f}}

\newcommand{\rmm}{{\mathrm m}}

\newcommand{\rmB}{{\mathrm B}}
\newcommand{\rmC}{{\mathrm C}}
\newcommand{\rmD}{{\mathrm D}}

\newcommand{\rmI}{{\mathrm I}}
\newcommand{\rmJ}{{\mathrm J}}

\newcommand{\rmM}{{\mathrm M}}

\newcommand{\rmT}{{\mathrm T}}

\newcommand{\rmV}{{\mathrm V}}

\newcommand{\Kliminf}{K\kern-3pt-\kern-2pt\mathop{\rm lim\,inf}\limits}  % Kuratowski liminf di insiemi
\newcommand{\argmin}{\mathop{\rm argmin}\limits}   % argmin
\renewcommand{\d}{{\mathrm d}}
\newcommand{\dt}{{\d t}}

\newcommand{\restr}[1]{\lower3pt\hbox{$|_{#1}$}}
\newcommand{\topref}[2]{\stackrel{\eqref{#1}}#2}
\newcommand{\Leb}[1]{{\mathscr L}^{#1}}      % Misura di Lebesgue
\newcommand{\down}{\downarrow}              %frecce in su e in giu nei limiti
\newcommand{\up}{\uparrow}
\newcommand{\weakto}{\rightharpoonup}
\newcommand{\weaksto}{\rightharpoonup^*}

\newcommand{\eps}{\varepsilon}  
\newcommand{\nchi}{{\raise.3ex\hbox{$\chi$}}}
\newcommand{\Rd}{{\R^d}}

\newcommand{\forae}{\text{for a.e.}}

\newcommand{\forevery}{\text{for every }}
\usepackage{pdfsync}

\newcommand{\essJ}{\mathop{{\rm ess}\text{-}{\rm J}}\nolimits}
\newcommand{\AC}[3]{\mathrm{AC}^{#1}(#2;#3)}

\newcommand{\res}{\mathop{\hbox{\vrule height 7pt width .5pt depth 0pt
\vrule height .5pt width 6pt depth 0pt}}\nolimits} % macro per la restrizione

\newcommand{\BV}[2]{\mathrm{BV}([#1];#2)}
\newcommand{\Var}{\mathrm{Var}}

\newcommand{\mmd}[2]{|\d#1|_{#2}}
\newcommand{\comd}[2]{|#1'|_{#2}}
\newcommand{\jmd}[2]{|\rmJ #1|_{#2}}
\newcommand{\Cmd}[2]{|\rmC #1|_{#2}}
\newcommand{\Jmp}[1]{\mathrm{Jmp}_{#1}}
\newcommand{\pJmp}[1]{\text{\slshape Jmp}_{#1}}
\newcommand{\topto}{\stackrel\sigma\rightarrow}
\newcommand{\hinH}{h\in H}
\newcommand{\finsl}[3]{\mathrm{Length}_{#1,#2}[#3]}
\newcommand{\bicost}[3]{\Delta_{#1,#2}(#3)}
\newcommand{\tricost}[3]{\Delta_{#1,#2}(#3)}

\renewcommand{\top}{\sigma}

\renewcommand{\cT}{\rmI}

\begin{document}

%-------------------------------------------------------------------------
% editorial commands: to be inserted by the editorial office
%
%\firstpage{1} \volume{228} \Copyrightyear{2004} \DOI{003-0001}
%
%
%\seriesextra{Just an add-on}
%\seriesextraline{This is the Concrete Title of this Book\br H.E. R and S.T.C. W, Eds.}
%
% for journals:
%
%\firstpage{1}
%\issuenumber{1}
%\Volumeandyear{1 (2004)}
%\Copyrightyear{2004}
%\DOI{003-xxxx-y}
%\Signet
%\commby{inhouse}
%\submitted{March 14, 2003}
%\received{March 16, 2000}
%\revised{June 1, 2000}
%\accepted{July 22, 2000}
%
%
%
%---------------------------------------------------------------------------
%Insert here the title, affiliations and abstract:
%

\title[Variational convergence of gradient and rate-independent flows]{Variational convergence of gradient flows and rate-independent
  evolutions in metric spaces}

\author{Alexander Mielke}

\address{{Weierstra\ss-Institut,
    Mohrenstra\ss{}e 39, 10117 D--Berlin\\
    and\\
    Institut f\"ur
    Mathematik, Humboldt-Universit\"at zu
    Berlin\\
    Rudower Chaussee 25, D--12489 Berlin (Adlershof)\\
    Germany}}

\email{mielke\,@\,wias-berlin.de}

\author{Riccarda Rossi}
\address{{Dipartimento di Matematica, Universit\`a di
 Brescia\\
 via Valotti 9, I--25133 Brescia, Italy}}

\email{riccarda.rossi\,@\,ing.unibs.it}

\author{Giuseppe Savar\'e}
\address{{Dipartimento di Matematica ``F.\
Casorati'', Universit\`a di Pavia\\
 Via Ferrata 1, I--27100 Pavia, Italy}}

\email{giuseppe.savare\,@\,unipv.it}

%----------classification, keywords, date


\date{July 10, 2012}

\begin{abstract} 
  We study the asymptotic behaviour of families of gradient flows in a general
  metric setting, when the metric-dissipation potentials degenerate in
  the limit to a dissipation with linear growth. 
  
  We present a
  general variational definition of BV solutions to metric evolutions, showing the
  different characterization of the solution in the absolutely
  continuous regime, on the singular Cantor part, and along the jump
  transitions.
  By using tools of metric analysis, BV functions and blow-up by time
  rescaling, we show that this variational notion is stable with respect to a wide class
  of perturbations involving energies, distances, and dissipation
  potentials.

  As a particular application, we show that BV solutions 
  to rate-independent problems arise naturally as a limit of
  $p$-gradient flows, $p>1$, when the exponents $p$ converge to $1$.
\end{abstract}
\thanks{R.R. and G.S. have been partially supported by
PRIN08 grant from MIUR for the project \emph{Optimal transport theory,
  geometric and functional inequalities, and applications}.}
\maketitle

%%%%%%%%%%%%%%%%%%%%%%%%%%%%%%%%%%%%%%%%%%%%%%%%%%%%%%%%%%%%
%                                                         % Body
%%%%%%%%%%%%%%%%%%%%%%%%%%%%%%%%%%%%%%%%%%%%%%%%%%%%%%%%%%%%

\section{Introduction}
\label{sec:intro}
The aim of this paper is to study the asymptotic behaviour of the
solutions to a sequence of gradient flows (in a suitable metric setting),
when the governing energies and metric-dissipation potentials 
give raise in the limit to a
rate-independent evolution or, more generally, to an evolution
driven by a dissipation potential with linear growth.

\subsection*{A finite dimensional example: superlinear dissipation
  potentials and absolutely continuous gradient flows.}
In order to explain the problem, let us start from a simple example 
in a finite dimensional manifold $\XX$
(see e.g.\ the motivating discussion in \cite{Mielke-Rossi-Savare09}).
We fix a time interval $[0,T]$, we denote by $\QQ$ the product space
$\QQ=[0,T]\times \XX$, and we consider a sequence of smooth energies
$\cE_h:\QQ\to\R$ indexed by
$h\in \N$. We are also given a sequence of smooth
dissipation potentials $\cR_h:\mathrm{T}\XX\to [0,\infty)$ 
of the form
$$\text{$\cR_h(u,\dot u)=\psi_h(\|\dot u\|_{u,h})$ \quad
  where $\|\cdot\|_{u,h}$ are norms
on the tangent space $\mathrm T_u\XX$}$$
smoothly depending on $u\in
\XX$ and 
$$\text{$\psi_h:[0,\infty)\to[0,\infty)$ are $\rmC^1$ convex functions with
superlinear growth}.$$
Typical examples are
\begin{equation}
  \label{eq:55}
  \begin{aligned}
    \psi_h(v)&=\frac 1{p_h}\,v^{p_h}\quad&&\text{with }p_h>1, \\
    %&&\lim_{h\to\infty}p_h=1
    \psi_h(v)&=v+ {\eps_h}\,v^p\quad&&\text{with }p>1,\ \eps_h>0.%\quad &&\lim_{h\to\infty}\eps_h=0.
  \end{aligned}
\end{equation}
For given initial data $\bar u_h\in \XX$ we can consider the solutions
$u_h:[0,T]\to \XX$ of the Cauchy problem for the doubly nonlinear differential equations
\begin{equation}
  \label{eq:86}
  \rmD_{\dot u} \cR_h(u_h(t),\dot u_h(t))+\rmD_u\cE_h(t,u_h(t))=0\quad 
  \text{in }\rmT^*\XX,%\ t\in [0,T],
  \quad
  u_h(0)=\bar u_h.
\end{equation}
In \eqref{eq:86} the parameter $h\in \N$ affects the limit behaviour
of the initial data $\bar u_h$, of the energies $\cE_h$ in $\QQ$, of the norms $\|\cdot\|_{\cdot,h}$
on $\rmT\XX$, and of the
dissipation potentials $\psi_h$ on $[0,\infty)$. Assuming that (in a
suitable sense that we will describe later on) $\bar
u_h\to\bar u$, $\cE_h\to\cE$, $\|\cdot\|_{u,h}\to \|\cdot\|_{u}$, $\psi_h\to\psi$ as
$h\to\infty$, it is then natural to investigate if a limit curve $u$
(possibly up to subsequence) of the
solutions $(u_h)_h$ still satisfies the corresponding limit equation
of \eqref{eq:86}.

\subsection*{BV solutions to rate-independent evolutions.}
We want to address here the singular situation when the limit
dissipation potential $\psi$ loses the superlinear growth;
let us focus here on the $1$-homogeneous case when
\begin{equation}
  \label{eq:87}
  \lim_{h\to\infty}\psi_h(v)=\psi(v):=\sfL \, v,\quad\text{for some $\sfL>0$}
\end{equation}
corresponding e.g.\
to $\lim_{h\to\infty}p_h=1,$ or $\lim_{h\to\infty}\eps_h=0$ in
\eqref{eq:55} (in that cases $\sfL=1$).
The limit problem is formally the differential inclusion
\begin{equation}
  \label{eq:86bis}
  \sfL\,\partial_{\dot u} \cR(u(t),\dot u(t))+\rmD_u\cE(t,u(t))\ni 0\quad \text{in
  }\rmT^*\XX,\ t\in [0,T],\quad
  u(0)=\bar u,
\end{equation}
where the presence of the subdifferential $\partial_{\dot u}\cR$ is
motivated by the lack of differentiability of the norm $\cR(u,\dot
u)=\|\dot u\|_u$ at $\dot u=0$. Since $\cR(u,\cdot)$ is $1$-positively
homogeneous, \eqref{eq:86bis} describes a rate-independent evolution and 
its solutions 
exhibit a different behavior with respect to the viscous flows
\eqref{eq:86}. In particular, jumps can occur even for smooth
energies $\cE$ and various kinds of solutions have been proposed in the
literature
(we refer to % see
% \cite{Mielke-Theil-Levitas02,Mainik-Mielke05,Mielke05,Mielke-Rossi-Savare09,Mielke-Rossi-Savare10}
the surveys \cite{Mielke05}, the overall presentation in \cite{Mielke11} and the references therein).
Here we focus on the notion of BV solution, proposed in 
\cite{Mielke-Rossi-Savare09,Mielke-Rossi-Savare10}:
for the sake of simplicity, in this introductory section 
we consider the simplest case of 
a piecewise
smooth curve $u$ with a finite number of jump points
$\rmJ_u=\{t_1,t_2,\cdots t_n\}\subset [0,T]$; $u(t_{i\,\pm})$ will
denote
the left and the right limit of $u$ at each $t_i$ (see also
\cite{Rossi-Savare12} for an explicit characterization in a one-dimensional setting)

In this case a BV solution $u$ can be characterized by two 
conditions:
\begin{enumerate}[(BV1)]
\item In each interval $(t_{i-1},t_i)$ the velocity
  vector field $\dot u$ satisfies the differential inclusion
  \eqref{eq:86bis}, which 
  yields 
   in particular the local stability condition
  \begin{equation}
    \label{eq:100}
    \cF(t,u(t))\le \sfL\quad \text{for every
    }t\in [0,T]\setminus \rmJ_u
  \end{equation} 
  and the energy dissipation 
  \begin{equation}
    \label{eq:104}
    -\frac \d{\d t}\cE(t,u(t))+\cP(t,u(t))=\sfL\, \|\dot
    u(t)\|_{u(t)}\quad\text{in }(t_{i-1},t_i),
  \end{equation}
  where $    \cP(t,u):=\frac\partial{\partial t}\cE(t,u)$ and
$\cF$ denotes the dual norm of the (opposite) differential of
  the energy, 
  \begin{equation}
    \label{eq:107}
    \cF(t,u):=\|\rmD_u\cE(t,u)\|^*_{u}=\sup\Big\{
    {-\langle\rmD_u\cE(t,u),v\rangle}:\|v\|_{u}\le 1\Big\}.
  \end{equation}
  It turns out that in the smooth regime \eqref{eq:100} and
  \eqref{eq:104} are equivalent 
  to \eqref{eq:86bis}.
\item  At each jump point $t_i$ it is possible to find an optimal
  transition path $\vartheta_i:[r_{i-},r_{i+}]\to \XX$, $r_{i-}\le 0\le r_{i+}$,
  such that $\vartheta_i(r_{i\pm})=u(t_{i\,\pm})$,
  $\vartheta_i(0)=u(t_i)$, $\cF(r,\vartheta_i(r))\ge \sfL$ in
  $[r_{i-},r_{i+}]$, and
  %\begin{equation}
    \begin{align}
      \label{eq:105}
      &\int_{r_{i-}}^{r_{i+}}\cF(r,\vartheta_i(r))\|\dot\vartheta_i(r)\|_{\vartheta_i(r)}\,\d
      r=\cE(t_i,u(t_{i\,-}))-\cE(t_i,u(t_{i\,+}))\\&\notag=
      \min\Big\{\int_{r_{i-}}^{r_{i+}}(\cF(r,\theta(r))\lor
      \sfL)\|\dot\theta(r)\|_{\theta(r)}\,\d
      r:% \theta:[r_-,r_+]\to \XX,\
      \theta(r_{i\,\pm})=u(t_{i\,\pm}),\ \theta(0)=u(t_i)\Big\}.
    \end{align}
%  \end{equation}
\end{enumerate}
Notice that the choice of the interval $[r_{i\,-},r_{i\,+}]$ is not
essential,
since the integrals in \eqref{eq:105} are invariant with respect to
monotone time rescaling. 
The minimum problem in
\eqref{eq:105} characterizes the minimal transition cost at each jump point $t_i$
to connect in $u(t_{i\,-})$ with $u(t_{i\,+})$ passing through
$u(t_i)$.
Such a cost is influenced both by the norms $\|\cdot\|_{u}$ and by the
slope $\cF$ of the energy: we will denote it by $\Delta_{t_i}(u(t_{i\,-}),u(t_i),u(t_{i\,+}))$.
\subsection*{Energy-dissipation inequalities.}
It is a remarkable fact, highlighted in
\cite{Mielke-Rossi-Savare09,Mielke-Rossi-Savare10},
that the refined structure given in (BV1,BV2) can be captured by simply imposing
the local stability condition \eqref{eq:100} and 
a single energy-dissipation inequality, namely
\begin{equation}
  \label{eq:101}
  \begin{aligned}
    \cE(T,u(T))+\sfL\int_0^T \|\dot u(t)\|_{u}\,\d t&+\sum_{i=1}^n
    \Delta_{t_i}(u(t_{i\,-}),u(t_i),u(t_{i\,+}))\\&\le \cE(0,\bar
    u)+\int_0^T \cP(t,u(t))\,\d t.\
  \end{aligned}
\end{equation}
It turns out that \eqref{eq:101} is in fact an identity, since the
opposite inequality is always satisfied along \emph{any} piecewise
smooth curve $u$. If \eqref{eq:101} holds, then $u$ is forced to
satisfy \eqref{eq:104} along its smooth evolution, and 
the optimal transition paths obtained by solving the minimum problem 
in \eqref{eq:105} provide the right energy balance between $u(t_{i\,\pm})$.

The link of \eqref{eq:101} with the gradient flow \eqref{eq:86}
becomes more transparent if, following
\cite{Ambrosio-Gigli-Savare08,Rossi-Savare06,Rossi-Mielke-Savare08,Mielke-Rossi-Savare12},
one notices that also \eqref{eq:86} can be formulated as a
energy-dissipation inequality.
In fact, setting as before
\begin{equation}
  \label{eq:90a}
  \cF_h(t,u):=\|\rmD_u\cE_h(t,u)\|^*_{u,h}% =\sup\Big\{
  % {-\langle\rmD_u\cE_h(t,u),v\rangle}:\|v\|_{u,h}\le 1\Big\}
  ,\qquad
  \cP_h(t,u):=\frac\partial{\partial t}\cE_h(t,u),
\end{equation}
it is not difficult to check (see the informal discussion in the next
section) that a $\rmC^1$ curve $u_h$ with $u_h(0)=\bar u_h$ satisfies
\eqref{eq:86} if and only if the $\psi_h$ energy-dissipation inequality holds
\begin{equation}
  \label{eq:98}
  \begin{aligned}
    \cE_h(T,u_h(T))&+\int_0^T \Big(\psi_h\big(\|\dot
    u_h(t)\|_{u_h,h}\big)+\psi_h^*\big(\cF_h(t,u_h(t))\big)\Big)\,\d
    t\\&\le \cE_h(0,\bar u_h)+\int_0^T \cP_h(t,u_h(t))\,\d t,
  \end{aligned}
\end{equation}
where $\psi_h^*$ is the Legendre transform of $\psi_h$. 

\paragraph{A more general formulation in metric spaces.}
Here we want to show that the metric-variational approach to gradient
flows 
and rate-independent problems
provides a natural framework to study this singular perturbation
problem
and suggests a robust and general strategy to pass to the limit in
a much more general setting 
where 
\begin{itemize}
\item[-] $\XX$ is a topological space endowed with a family of complete
  extended distances $\sfd_h$, 
\item[-] the terms like $\|\dot u_h\|_{u_h,h}$
  are replaced by the \emph{metric velocity} induced by $\sfd_h$, 
\item[-] the
  functions $\cF_h,\cP_h$ can be characterized as an \emph{irreversible
    couple of upper gradients} in terms of the behaviour of the
  energies $\cE_h$ along arbitrary absolutely continuous curves with
  values in $(\XX,\sfd_h)$, and 
\item[-] $\psi$ is a general metric dissipation
  function with linear growth.
\end{itemize}

Postponing to the next two sections a more precise review of
motivations and definitions, we just remark that whenever sufficiently
strong a
priori estimates are available to guarantee the pointwise convergence
of $u_h$ to some limit function $u\in \BV{0,T}{(\XX,\sfd)}$,
then the heart of the problem consists in deriving \eqref{eq:101} (in a
suitably extended form allowing countably many jumps and Cantor-like terms
in the metric velocity), starting from the viscous inequality
\eqref{eq:98}.
Assuming convergence in energy of the initial data, i.e.\
$\lim_{h\to\infty}\cE_h(0,\bar u_h)=\cE(0,\bar u)$,
some lower-upper semicontinuity conditions on $(\cE_h)_h$ and $(\cP_h)_h$ along arbitrary
sequences $(x_h)_h$ with equibounded energy and
converging to $x$ in a fixed reference topology $\sigma$ of $\XX$
are naturally suggested by the structure of the main inequalities
\eqref{eq:98} and \eqref{eq:101}:
\begin{equation}
  \label{eq:108}
  \left\{
    \begin{aligned}
      &\liminf_{h\to\infty}\cE_h(t,x_h)\ge \cE(t,x), \\&
      % \qquad
      \limsup_{h\to\infty}\cP_h(t,x_h)\le \cP(t,x),
  \end{aligned}
  \right.
  \qquad\text{whenever}\quad
  x_h\topto x\text{ in }\XX.
\end{equation}
The most challenging point is provided by the limit behaviour of the integral term
\begin{equation}
  \label{eq:102}
  \int_0^T  \Big(\psi_h\big(\|\dot
  u_h(t)\|_{u_h,h}\big)+\psi_h^*\big(\cF_h(t,u_h(t))\big)\Big)\,\d
  t,
\end{equation}
which has been typically studied by a clever re-parametrization
technique, introduced by \cite{Efendiev-Mielke06} and then extended 
in various directions by
\cite{Mielke-Rossi-Savare09,Mielke-Zelik12,Mielke-Rossi-Savare10}.
This approach leads to the notion of the so-called \emph{parametrized
  solutions} to the rate-independent evolution and the crucial
assumption concerns the validity of the $\Gamma$-$\liminf$ space-time
estimate for the slopes
\begin{equation}
  \label{eq:103}
  \liminf_{h\to\infty}\cF_h(t_h,x_h)\ge
  \cF(t,x)\quad\text{whenever}\quad t_h\to
  t,\ x_h\topto x.
\end{equation}
In the present paper we propose a different technique, which avoids parametrized
solutions and thus allows for more general non-homogenous dissipation potentials
like
\begin{equation}
  \label{eq:112}
  \begin{aligned}
    \psi(v):=&\int_0^v (r\land \sfL)\,\d r=
    \begin{cases}
      \frac 12 v^2&\text{if }0\le v\le \sfL,\\
      \sfL\,v-\frac 12\sfL^2&\text{if }v\ge \sfL,
    \end{cases}
    \\
    %\quad
  %\psi(v):=\frac 1{p^*}v-\frac 1p[(1-v)-(1-v)^p]_+,\quad
    \psi(v):=&(1+v^2)^{1/2}.
  \end{aligned}
\end{equation}
Our approach involves weak convergence of measures to deal with
concentrations of the time derivative and blow-up
around jump points of the limit solution to recover the variational
structure of the transition.
In this way, an
easier rescaling is sufficient to construct the optimal transition
paths (see \eqref{eq:105}) from the converging family $(u_h)_h$
and to obtain the BV energy-dissipation inequality
\eqref{eq:101}.

\paragraph{Particular cases.}
Let us remark that various particular cases of the present setting are
interesting by themselves and have been considered from many different
points of view.
\begin{enumerate}[(i)]
\item A first important case for applications is when $\XX$ is a
  Hilbert space, $\psi_h(v)=\frac 12 v^2$, and the norms $\|\cdot\|_{u,h}$ are independent of $h$
  and coincide with the norm $\|\cdot\|$ of $\XX$.
  In this case we are dealing with the convergence of gradient flows
  and a typical situation arises when $\cE_h(t,u)=\cE_h(u)-\langle
  \ell(t), u\rangle$. 
  It is well known, since the pioneering contributions of 
  \cite{Spagnolo67,Spagnolo68,DeGiorgi-Spagnolo73},
  that convexity (or $\lambda$-convexity for some
  $\lambda\in \R$ independent of $h$) of the energies makes it
  possible to reduce
  \eqref{eq:103}
  to the simpler Mosco-convergence \cite{Mosco69} of $\cE_h$ 
  (see e.g.\
  \cite{Attouch84} or \cite{Brezis73} for the connection with the
  graph-convergence of the differential operators). 
  The link between $\Gamma$-convergence
  of the energies and convergence of the gradient flows in a metric
  setting has been considered in 
  \cite{Ambrosio-Gigli-Savare08,Ambrosio-Savare-Zambotti09,Daneri-Savare-preprint10}.

\item Another relevant situation is when both the energies and the
  distances depend on $h$: in the quadratic case a convergence result
  can be deduced by a joint $\Gamma$-convergence, see e.g.\
  \cite{Savare-Visintin97,Pennacchio-Savare-ColliFranzone05,Peletier-Savare-Veneroni10}.
  The role of the $\Gamma$-$\liminf$ condition on 
  the slopes as in \eqref{eq:103} in general non-convex setting has been 
  clarified in \cite{Ortner05,Serfaty11}.
 A very general stability result has been given in \cite{Mielke-Rossi-Savare12}.
  An interesting example where the limit of gradient flows gives raise
  to a singular limit in a new geometry is discussed in
  \cite{Arnrich-Mielke-Peletier-Savare-Veneroni-preprint11}.
\item The particular case when the $h$-dependence affects only the
  dissipation potential $\psi$ and gives raise to a rate-independent
  problem in the limit has been studied in 
  \cite{Mielke-Rossi-Savare09,Mielke-Rossi-Savare10,Mielke-Rossi-Savare-preprint12}.
  The $\Gamma$-limit of rate-independent evolutions, 
  in the framework of energetic solutions, has been studied in
  \cite{MielkeRoubicekStefanelli08}.
\end{enumerate}
\paragraph{Plan of the paper.}
In the next section we give more details on the simple
finite-dimensional example we introduced before, in order to
motivate the abstract metric approach, whose setting is explained
in Section \ref{sec:preliminaries}.

Section \ref{sec:main} contains our main results, concerning
compactness (Theorem \ref{thm:compactness}) 
and convergence (Theorem \ref{thm:convergence}) of gradient flows in a general setting.
A few examples are briefly presented at the end of the paper.

\section{The metric formulation of gradient flows in a smooth setting}
\label{sec:smooth}
Let $\XX$ be the finite-dimensional differentiable manifold discussed
in the Introduction. 
\subsection*{Length and metric derivative.} 
Let us first recall that the Finsler structure $\|\cdot\|_{u,h}$ on
$\rmT\XX$ allows us to define the length of a smooth curve
$\sfu\in \rmC^1([r_0,r_1];\XX)$ by
\begin{equation}
  \label{eq:88}
  \mathrm{Length}_h[\sfu]:=\int_{r_0}^{r_1}
  \|\dot\sfu(r)\|_{\sfu(r),h}\,\d r
\end{equation}
and a distance
\begin{equation}
  \label{eq:44}
  \sfd_h(u_0,u_1):=\inf\Big\{\mathrm{Length}_h[\sfu]% \int_{r_0}^{r_1}
  % \|\dot\sfu(r)\|_{\sfu(r),h}\,\d r
  :\sfu\in
  \rmC^1([r_0,r_1];\XX),\ \sfu(r_i)=u_i\Big\},
\end{equation}
which still retains the information of the norms $\|\cdot\|_{u,h}$,
since
\begin{equation}
  \label{eq:89}
  \|\dot\sfu(r)\|_{\sfu(r),h}=\lim_{s\to
    r}\frac{\sfd_h(\sfu(s),\sfu(r))}{|s-r|}\quad
  \forevery \sfu\in \rmC^1([r_0,r_1];\XX).
\end{equation}
The limit in \eqref{eq:89} can be extended to the general
setting of absolutely continuous curves in metric spaces:
it is denoted by $|\dot\sfu|_{\sfd_h}(r)$ and it is called
\emph{metric derivative} of the curve $\sfu$, see Definition \ref{def:metric-derivative}.
\subsection*{Chain rule and irreversible upper gradients.}
A second crucial quantity is the dual norm of the opposite differential of the
energy % (in the present symmetric setting, the minus in front of
% $\rmD_u\cE_h$ does not play any role, but it suggests the right
% inequalities and it is essential
% if one wants to extend these result to a possible asymmetric framework, see
% e.g.\ \cite{Rossi-Mielke-Savare08}):
\begin{equation}
  \label{eq:90}
  %\cF_h(t,u):=
  \|\rmD_u\cE_h(t,u)\|^*_{u,h}=\sup\Big\{
  {-\langle\rmD_u\cE_h(t,u),v\rangle}:\|v\|_{u,h}\le 1\Big\}.
\end{equation}
Observe that the quantity in \eqref{eq:90} also has a nice characterization in terms of curves, since the function
$(t,u)\mapsto \|\rmD_u\cE_h(t,u)\|^*_{u,h}$ is minimal among the
functions $\cF_h:\QQ\to[0,\infty)$ satisfying
the chain-rule inequality
\begin{equation}
  \label{eq:95}
  -\frac \partial{\partial r}\cE_h(t,\sfu(r))\le \cF_h(t,\sfu(r))|\dot\sfu|_{\sfd_h}(r)
\end{equation}
along arbitrary curves $\sfu\in \rmC^1([r_0,r_1];\XX)$. If one wants
to allow for time
variation of the energy, it is natural to introduce the partial time
derivative $\frac\partial{\partial t}\cE_h(t,u)$,
% \begin{equation}
%   \label{eq:91}
%   \cP_h(t,u):=
% \end{equation}
so that \eqref{eq:95} is in fact equivalent
to
\begin{equation}
  \label{eq:92}
  -\frac \d{\d r}\cE_h(t(r),\sfu(r))+\frac\partial{\partial
    t}\cE_h(\sft(r),\sfu(r))
  % \cP_h(t(r),\sfu(r))
  \dot \sft(r)\le \cF_h(\sft(r),\sfu(r))|\dot\sfu|_{\sfd_h}(r)
\end{equation}
along arbitrary regular curves $r\mapsto
(\sft(r),\sfu(r))\in \QQ$. 
If we only consider nondecreasing time parametrizations $r\mapsto
t(r)$, 
and we integrate \eqref{eq:92} along arbitrary intervals $[r_0,r_1]$, 
we see that the map $(t,u)\mapsto \frac\partial{\partial t}\cE(t,u)$ is maximal among all the functions
$\cP_h:\QQ\to\R$ satisfying 
\begin{equation}
  \label{eq:93}
  \begin{aligned}
    \cE_h(\sft(r_0),\sfu(r_0))&+\int_{r_0}^{r_1}\cP_h(\sft(r),\sfu(r))\dot
    \sft(r)\,\d r
    \\&
    \le \cE_h(\sft(r_1),\sfu(r_1))+\int_{r_0}^{r_1}
    \cF_h(\sft(r),\sfu(r))|\dot\sfu|_{\sfd_h}(r)\,\d r.
  \end{aligned}
\end{equation}
In fact, let us suppose that $\cF_h,\cP_h$ are continuous 
functions
satisfying \eqref{eq:93} along arbitrary regular curves with $\dot
\sft(r)\ge0$:
it would not be difficult to check that this property is equivalent to
\begin{equation}
  \label{eq:94}
  \left\{\begin{aligned}
      \cP_h(t,u)&\le \frac\partial{\partial t}\cE(t,u),\\
      \cF_h(t,u)&\ge
    \|\rmD_u\cE_h(t,u)\|^*_{u,h}
  \end{aligned}
  \right.
  \qquad\text{for every
  }(t,u)\in \QQ.
\end{equation}
If \eqref{eq:94} holds, we say that the couple $(\cF_h,\cP_h)$ is an \emph{irreversible
upper gradient} for the energy $\cE_h$ with respect to the distance
$\sfd_h$, see Definition \ref{def:upper-gradients},
and $(\XX,\sfd_h,\cE_h,\cF_h,\cP_h)$ is an upper-gradient evolution system.

This definition is the natural adaptation to
time-dependent functionals of the well-known notion of upper-gradient
in the frame of 
analysis in metric spaces (see \cite{Cheeger00,Ambrosio-Gigli-Savare08});
the interesting fact is that \eqref{eq:93} only involves the notion of
absolutely continuous curves in $(\XX,\sfd_h)$.
\subsection*{$\psi$-gradient flows and energy-dissipation
  inequality.}
The distinguished role of gradient flows with respect to \eqref{eq:92}
can be easily seen by recalling the Fenchel duality 
\begin{equation}
  \label{eq:97}
  -\rmD_{v}\cR_h(u,v)=f\quad\Leftrightarrow\quad
  -\langle f,v\rangle=\|v\|_{u,h}\|f\|^*_{u,h}=
  \psi_h(\|v\|_{u,h})+\psi^*_h(\|f\|^*_{u,h}),
\end{equation}
where $\psi^*$ is the Legendre transform of $\psi$:
\begin{equation}
  \label{eq:99}
    \psi^*(f)=\sup_{v\ge 0}\Big(f\,v-\psi(v)\Big).
\end{equation}
A crucial feature of Fenchel duality is that for \emph{every} couple
$(v,f)\in \rmT_u\XX\times \rmT_u^*\XX$ one has the inequality
\begin{equation}
  \label{eq:133}
  -\langle f,v\rangle\le \|v\|_{u,h}\|f\|^*_{u,h}\le 
  \psi_h(\|v\|_{u,h})+\psi^*_h(\|f\|^*_{u,h}),
\end{equation}
so that in order to check the identity $-\rmD_{v}\cR_h(u,v)=f$ it is
sufficient to prove the opposite inequality, i.e.
\begin{equation}
  \label{eq:97opposite}
  -\langle f,v\rangle\ge
  \psi_h(\|v\|_{u,h})+\psi^*_h(\|f\|^*_{u,h})
  \quad\Rightarrow\quad
  -\rmD_{v}\cR_h(u,v)=f.
\end{equation}
Taking into account these remarks and observing that 
we have the chain rule
\begin{displaymath}
  -\langle
   \rmD_u\cE_h(t,u_h(t)),\dot u_h(t)\rangle=-\frac \d{\d
      t}\cE_h(t,u_h(t))+\frac\partial{ \partial t}\cE_h(t,u_h(t)),
\end{displaymath}
we deduce that $u_h$ solves \eqref{eq:86} if and only if
\begin{equation}
  \label{eq:96}
  \begin{aligned}
    -\frac \d{\d t}\cE_h(t,u_h(t))&+\frac\partial{ \partial
      t}\cE_h(t,u_h(t))\\&\ge \psi_h\big(\|\dot
    u_h(t)\|_{u_h,h}\big)+\psi_h^*\big(
    \|\rmD_u\cE_h(t,u_h(t))\|^*_{u_h,h}\big).
  \end{aligned}
\end{equation}
Since, as we already noticed in \eqref{eq:133}, the opposite
inequality is always true, we immediately see that it is sufficient to
impose the integrated version of \eqref{eq:96} in $(0,T)$:
\begin{equation}
  \label{eq:98T}
  \begin{aligned}
    \cE_h(T,u_h(T))&+\int_0^T \Big( \psi_h\big(\|\dot
    u_h(t)\|_{u_h,h}\big)+\psi_h^*\big(
    \|\rmD_u\cE_h(t,u_h(t))\|^*_{u_h,h}\big)\Big)\,\d t\\&\le
    \cE_h(0,\bar u_h)+\int_0^T \frac\partial{\partial
      t}\cE_h(t,u_h(t))\,\d t.
  \end{aligned}
\end{equation}
Now we can make the last step: instead of 
looking for curves satisfying \eqref{eq:98T}, we reinforce it by
replacing $\partial_t \cE_h$ and $\|\rmD_u\cE_h\|^*_{u_h,h}$ 
with a couple $\cP_h,\cF_h$ of irreversible upper gradients satisfying \eqref{eq:93}.
Since $\psi$ and $\psi^*$ are nondecreasing maps and \eqref{eq:94}
holds, it is immediate to see that if a curve $u\in \rmC^1([0,T];\XX)$
satisfies the $\psi$-$\psi^*$ energy-dissipation inequality 
\begin{equation}
  \label{eq:98FP}
  \begin{aligned}
    \cE_h(T,u_h(T))&+\int_0^T \Big(\psi_h\big(|\dot
    u_h|_{\sfd_h}(t)\big)+\psi_h^*\big(\cF_h(t,u_h(t))\big)\Big)\,\d t
    \\& \le \cE_h(0,\bar u_h)+\int_0^T \cP_h(t,u_h(t))\,\d t
  \end{aligned}
\end{equation}
(see also the next Definition \ref{def:ED}),
then it also satisfies \eqref{eq:98T} and by the argument above it
satisfies \eqref{eq:86}; moreover, along the curve we find a
posteriori
\begin{displaymath}
  \|\rmD_u\cE_h(t,u_h(t))\|^*_{u_h,h}=\cF_h(t,u_h(t)),\qquad
  \frac\partial{\partial
      t}\cE_h(t,u_h(t))=\cP_h(t,u_h(t))
\end{displaymath}
for every $t\in [0,T]$.
We thus have seen that \eqref{eq:98FP} for a couple $(\cF_h,\cP_h)$ of
irreversible upper gradients provides a natural metric definition of $\psi$-gradient
flow, which can be immediately extended to a metric framework.

\subsection*{Marginal functionals and conditional time derivative of
  the energy.}
In order to motivate the even more general definition of evolution system considered in 
Section \ref{subsec:metric}, where the power functional $\cP$ 
\emph{can also depend on a further variable $\sF$ satisfying the
  constraint $\sF(t)\ge \cF(t,u(t))$}, let us consider 
a non smooth situation, where $\cE$ is a \emph{marginal functional}:
it means that $
\cE$ results from a minimization of the form
\begin{equation}
  \label{eq:27}
  \cE(t,u):=\min_\eta\big\{\cI(t,u,\eta):\eta\in \KK\big\},
\end{equation}
where $\KK$ is a compact topological space and $\cI:\QQ\times\KK\to \R$ is a continuous function such that 
$\cI(\cdot,\cdot,\eta)\in \rmC^1(\QQ)$ for every $\eta\in \KK$ with
uniformly continuous derivatives.

Even if each single functional $\cI(\cdot,\eta)$ is regular, $\cE$ is
not $\rmC^1$ in general. Referring to \cite{Mielke-Rossi-Savare12} for
a more detailed discussion, we recall here that setting
\begin{equation}
  \label{eq:136}
  M(t,u):=\argmin\cI(t,u,\cdot)=\Big\{\eta\in \KK: \cI(t,u,\eta)=\cE(t,u)\Big\}
\end{equation}
it is natural to replace the smooth differential equation
\eqref{eq:86} with the differential inclusion in $[0,T]$
\newcommand{\rmmm}{{\scriptscriptstyle \rmm}}
\begin{equation}
  \label{eq:86marginal}
  \rmD_{\dot u} \cR_h(u_h(t),\dot u_h(t))+\rmD^{\rmmm}_u\cE(t,u_h(t))\ni 0\quad \text{in
  }\rmT^*\XX,
  %\ t\in [0,T],
  \quad
  u_h(0)=\bar u_h,
\end{equation}
where, just for the purposes of this section, $\rmD^\rmmm\cE$ denotes the so-called \emph{marginal
  differential} of $\cE_h$, i.e.
\begin{align*}
  \rmD^\rmmm\cE(t,u):=\Big\{(\sfp,\sfw)\in \R\times \rmT^*_u\XX:
  \sfp&=\partial_t \cI(t,u,\eta),\\
  \sfw&=\rmD_u\cI(t,u,\eta)\ \text{for some }\eta\in M(t,u)\Big\}
\end{align*}
and $\rmD^\rmmm_u\cE$ is its projection onto the second component,
\begin{equation}
  \label{eq:151}
  \rmD^\rmmm_u \cE(t,u):=\Big\{w\in
  \rmT^*_u\XX:w=\rmD_u\cI(t,u,\eta)\quad\text{for some }\eta\in M(t,u)\Big\}.
\end{equation}
If we want to differentiate the energy along a regular curve
$r\mapsto(\sft(r),\sfu(r))$ as in \eqref{eq:92} we get for a.a.\ $r$
\begin{equation}
  \label{eq:58}
  -\frac \d{\d r}\cE(\sft(r),\sfu(r))+\sfp(r)
  %(t,\sfu(r),\sfw(r))
  \,\dot \sft(r)=
  -\langle \sfw(r),\dot\sfu(r)\rangle\le 
  \|\sfw(r)\|_{\sfu,h}^*\,|\dot\sfu|_{\sfd_h}(r),
\end{equation}
where $(\sfp(r),\sfw(r))$ is an arbitrary selection in %$\in \rmT^*_{\sfu(r)}\XX$ 
%is an arbitrary selection in $\partial_u
%\cE(t,\sfu(r))$ 
$\rmD^\rmmm\cE(\sft(r),\sfu(r))$.
% and $\sfp(r)\in \R$ are arbitrary values satisfying
% \begin{equation}
%   \label{eq:59}
%   \sfw(r)=\rmD_u\cI(t,\sfu(r),\eta),\quad
%   \sfp(r)=\frac \partial{\partial t}\cI(t,\sfu(r),\eta)\quad\text{for
%     some } \eta\in
%   M(\sft(r),\sfu(r)).
%   % \ \text{such that } \sfw(r)=\rmD_u\cI(t,\sfu(r),\eta)
% \end{equation}
% \begin{equation}
%   \label{eq:59}
%   \sfp(r)=\frac \partial{\partial t}\cI(t,\sfu(r),\eta)\quad\text{for
%     some } \eta\in
%   M(t,u) \ \text{such that } \sfw(r)=\rmD_u\cI(t,\sfu(r),\eta)
% \end{equation}
Setting 
\begin{align}
  \label{eq:137}
  \cF_h(t,u):=&\min\Big\{\|\sfw\|_{u,h}:\sfw\in
  \rmD_u^\rmmm(t,u)\Big\},
  %\Big\{\|\rmD_u\cI(t,u,\eta)\|_{u,h}:\eta\in   M(t,u)\Big\},
  \\
  \label{eq:152}\cP_h(t,u,f):=&\max\Big\{\sfp:(\sfp,\sfw)\in \rmD^\rmm\cE(t,u),\
  \|\sfw\|_{u,h}\le f\Big\},
  %\Big\{\partial_t\cI(t,u,\eta):\eta\in M(t,u),\   \|\rmD_u\cI(t,u,\eta)\|_{u,h}\le f\Big\},
\end{align}
it is easy to check that for every $\sF(r)\ge \cF_h(\sft(r),\sfu(r))$ we
have
\begin{equation}
  \label{eq:147}
   -\frac \d{\d r}\cE(\sft(r),\sfu(r))+\cP_h(\sft(r),\sfu(r),\sF(r))\dot
   \sft(r)\le 
   \sF(r)\,|\dot\sfu|_{\sfd_h}(r).
\end{equation}
Conversely, if a curve $[0,T]\ni t\mapsto u_h$  satisfies the $\psi_h$
energy-dissipation inequality
\begin{equation}
  \label{eq:148}
  -\frac \d{\dt}\cE(t,u(t))+\cP_h(t,u(t),\sF_h(t))
  \ge
   \psi_h\big(|\dot u|_{\sfd_h}(t)\big)+\psi_h^*\big(\sF_h(t)\big)
\end{equation}
for a.a.\ $t\in (0,T)$ and for some $\sF_h(t)\ge \cF_h(t,u_h(t))$, we get 
% $w_h(t)\in \partial_u
% \cE(t,u_h(t))$ so that 
% \begin{equation}
%   \label{eq:149}
%   \begin{aligned}
%     \|w_h(t)\|_{u,h}^*&\,|\dot
%     u_h|_{\sfd_h}(t)-\cP_h(t,u_h(t),w_h(t))\\&=\min\Big\{f\in \partial_u\cE(t,u_h(t)):
%     \|f\|_{u,h}^*\,|u_h|_{\sfd_h}(t)-\cP_u(t,u_h(t),f)\Big\}
%   \end{aligned}
% \end{equation}
by \eqref{eq:58} and \eqref{eq:147} 
\begin{gather*}
  %\label{eq:150}
  -\langle \sfw,\dot
  u_h(t)\rangle -\sfp+\cP_h(t,u(t),\sF_h(t))\ge 
    \psi_h\big(|\dot u|_{\sfd_h}(t)\big)+\psi_h^*\big(\sF_h(t)\big)
    \\\text{for every }
    (\sfp,\sfw)\in \rmD^\rmmm\cE(t,u_h(t)).
\end{gather*}
Choosing in particular a couple $(\bar \sfp,\bar\sfw)$ attaining the
maximum in \eqref{eq:152} for $f:=\sF_h(t)$, we obtain
\begin{displaymath}
   -\langle \bar\sfw,\dot
  u_h(t)\rangle\ge \psi_h\big(|\dot
  u|_{\sfd_h}(t)\big)+\psi_h^*\big(\sF_h(t)\big)
  \ge \psi_h\big(|\dot
  u|_{\sfd_h}(t)\big)+\psi_h^*\big(\bar\sfw\big),
\end{displaymath}
which eventually yields by \eqref{eq:97}
\begin{displaymath}
  -\rmD_{\dot u}\cR_h(u_h(t),\dot u_h(t))=\bar\sfw\in
  \rmD_u^\rmmm\cE(t,u_h(t)),\quad
  \sF_h(t)=\|\bar \sfw\|_{u_h,h},
\end{displaymath}
so that $u_h$ solves \eqref{eq:86marginal}.

\subsection*{Towards a general form of chain rule and
  energy-dissipation inequalities.}
Notice that we were able to formulate the non-smooth differential
inclusion \eqref{eq:86marginal} in a metric variational form by
looking for a
chain-rule inequality with the more general structure given by \eqref{eq:147}: this
will be reflected in the definition \ref{def:upper-gradients}
 of irreversible upper gradients.

The differential inclusion is then characterized by 
the $\psi_h$ energy-dissipation inequality \eqref{eq:148}:
its metric formulation will be considered in Definition
\ref{def:ED} in the superlinear case and in Definition \ref{def:EDBV} in
the case of a metric dissipation $\psi$ with linear growth.
% one can see that 
% $(\XX,\sfd,\cE,\cF,\cP)$ is an upper gradient evolution system.
% $\psi_h$-gradient flows in the superlinear case can be obtained by
% applying the results of \cite{Mielke-Rossi-Savare12}.

It is then natural to investigate the stability of inequality \eqref{eq:98T}
with respect to perturbations of the parameter $h$. 
One of the most difficult points is to guess how to state 
\eqref{eq:98FP} when 
the metric dissipation functional $\psi$ has only a linear growth, and
therefore one expects a solution in $\BV{0,T}{(\XX,\sfd)}$.
We have already discussed in the introduction the case of a piecewise
smooth curve, but a robust theory should allow for general BV curves,
possibly exhibiting countably many jumps and a metric derivative with a
singular Cantor part.
The correct treatment of this case will be discussed in the next section.

\section{The metric setting and preliminary results}
\label{sec:preliminaries}

\subsubsection*{Complete extended distances}
Let $\XX$ be a given set;
an \emph{extended distance} on $\XX$ is a map $\sfd:\XX\times\XX\to
[0,\infty]$ satisfying
\begin{align*}
  &\sfd(x,y)=0&&\quad\text{if and only if }x=y,\\
  &\sfd(x,y)=\sfd(y,x)&&\quad\text{for every }x,y\in \XX,\\
  &\sfd(x,z)\le \sfd(x,y)+\sfd(y,z)&&\quad\text{for every }x,y,z\in \XX.
\end{align*}
We say that $(\XX,\sfd)$ is an extended metric space. Most of the definitions concerning metric spaces generalize verbatim to extended metric spaces,
in particular it makes perfectly sense to speak about a complete extended metric space.

\subsection{BV, absolutely continuous curves, and metric
  derivative}
Let $(\XX,\sfd)$ be an extended metric space. 
\begin{definition}[Absolutely continuous curves and metric
  derivatives]
  \label{def:metric-derivative}
  %Let\\ $p\in [1,\infty]$. 
  We say that a curve $u:[a,b]\to \XX$ is absolutely continuous (a.c.\
  for short) and belongs to
  $\AC {}{a,b}{(\XX,\sfd)}$ %(we omit the index $p$ in the case $p=1$)
  if there exists $m\in L^1(a,b)$ such that 
  \begin{equation}
    \label{eq:2}
    \sfd(u(s),u(t))\le \int_s^t m(r)\,\d r\quad\text{for every }a\le
    s<t\le b.
  \end{equation}
  If $u\in \AC {}{a,b}{(\XX,\sfd)}$ then the limit
  \begin{equation}
    \label{eq:3}
    |\dot
    u|_{\sfd}(t):=\lim_{\tau\down0}\frac{\sfd(u(t+\tau),u(t))}{|\tau|}\quad\text{exists
      for $\Leb 1$-a.a.\ $t\in (a,b)$},
  \end{equation}
  it satisfies $|\dot u|_\sfd\le m$ $\Leb 1$-a.e.\ in $(a,b)$, belongs
  to $L^1(a,b)$, and it is called \emph{metric derivative of $u$;} 
  $|\dot u|_\sfd$
  provides the minimal function $m$ such that \eqref{eq:2} holds.
\end{definition}
The (pointwise) $\sfd$-variation of $u:[a,b]\to \XX$ in an interval
$[\alpha,\beta]\subset [a,b]$ is defined by
\begin{equation}
  \label{eq:8}
  \mathrm{Var}_\sfd(u;[\alpha,\beta]):=\sup\Big\{\sum_{j=1}^n
  \sfd(u(t_j),u(t_{j-1})):\alpha=t_0<\cdots<t_n=b\Big\}.
\end{equation}
We say that $u\in \BV{a,b}{(\XX,\sfd)}$ if
$\mathrm{Var}_\sfd(u;[a,b])<\infty$ and $u$ takes values
in a complete subset of $(\XX,\sfd)$; in this case, 
$u$ admits left and right limits (denoted by $u(t_-)$ and
$u(t_+)$) at every point of $[a,b]$ and we adopt the convention to
extend
$u$ to $\R\setminus[a,b]$ by setting
\begin{equation}
  \label{eq:140}
  u(t):=
  \begin{cases}
    u(a)&\text{if }t<a,\\
    u(b)&\text{if }t>b,
  \end{cases}
  \quad\text{so that}\quad
  u(a_-):=u(a),\ u(b_+):=u(b).
\end{equation}
The \emph{pointwise jump set} and the \emph{essential jump set} of $u$
are 
defined by
\begin{equation}
  \label{eq:9}
  \begin{aligned}
    \rmJ_u:=&\big\{t\in [a,b]:u(t)\neq u(t_-)\ \text{or }u(t)\neq
    u(t_+)\big\}
    \\
    \essJ_u:=&\big\{t\in [a,b]:u(t_-)\neq
    u(t_+)\big\},
  \end{aligned}
\end{equation}
and satisfy the obvious inclusion $\essJ_u\subset \rmJ_u$.
If $u\in \BV{a,b}{(\XX,\sfd)}$ we denote by $\rmV_u:\R\to [0,\infty)$
the bounded monotone function
\begin{equation}
  \label{eq:10}
  \rmV_u(t):=
  \begin{cases}
    0&\text{if }t< a,\\
    \Var_\sfd(u;[a,t])&\text{if }t\in [a,b],\\
    \Var_\sfd(u;[a,b])&\text{if }t>b,
  \end{cases}
\end{equation}
and by $\mmd u\sfd$ its distributional derivative: $\mmd u\sfd$ is a finite
measure in $\R$ supported in $[a,b]$, and we can decompose it 
as the sum of a diffuse part and a jump part
\begin{equation}
  \label{eq:11}
  \begin{gathered}
    \mmd u\sfd=\comd u\sfd+\jmd u\sfd,\quad \jmd u\sfd=\mmd u\sfd\res
    \rmJ_u,
    \\
    \comd u\sfd(\{t\})=0\quad\forevery t\in \R,
  \end{gathered}
\end{equation}
where $\res$ denotes the restriction of a measure to a Borel set; 
thus $\jmd u\sfd$ is concentrated on the (at most) countable jump set
$\rmJ_u$ and
\begin{equation}
  \label{eq:13}
  \jmd u\sfd(\{t\})=\sfd(u(t_-),u(t))+\sfd(u(t),u(t_+))\quad\forevery y\in \rmJ_u.
\end{equation}
The Lebesgue decomposition of the diffuse part $\comd u\sfd$ can be
written as 
\begin{equation}
  \label{eq:113}
  \comd u\sfd=|\dot u|_\sfd\,\Leb 1+\Cmd u\sfd,\quad \text{with $|\dot
    u|_\sfd$ given by \eqref{eq:3} and the Cantor part }\Cmd u\sfd\perp\Leb 1.
\end{equation}
We obtain
\begin{align}
  \label{eq:12}
    \Var_\sfd(u;[\alpha,\beta])&=\int_\alpha^\beta \d\,\comd u\sfd+\Jmp\sfd(u;[\alpha,\beta])\\&=
    \int_\alpha^\beta |\dot u|_\sfd(t) \,\d t+ \int_\alpha^\beta \d\,\Cmd u\sfd+\Jmp\sfd(u;[\alpha,\beta])
\end{align}
where for every subinterval $[\alpha,\beta]\subset [a,b]$ 
\begin{displaymath}
  \Jmp\sfd(u;[\alpha,\beta]):=\sfd(u(\alpha),u(\alpha_+))+\sum_{t\in
    \rmJ_u\cap (a,b)}\jmd u\sfd(\{t\})+\sfd(u(\beta_-),u(\beta)).
\end{displaymath}
\subsection{Metric evolution systems, irreversible upper gradients and
  $\psi$-gradient flows}
\label{subsec:metric}
Let $(\XX,\sfd)$ be a complete extended metric space and $[0,T]$ a
fixed time interval of $\R$. We denote by $\QQ$ the product space
$[0,T]\times \XX$ and we say that an a.c.\ curve 
$\sfq=(\sft,\sfu):[\alpha,\beta]\to \QQ$ is \emph{time-ordered} if 
$\sft$ is
non decreasing. 

If $I$ is some interval of $\R$, $\rmB_+(I)$ (resp.\ $\rmM_+(I)$) 
will denote the collections of Borel (resp.\ $\Leb 1$-measurable) maps defined in
$I$ with values in $[0,+\infty]$.
We say that a map $G:\QQ\to \overline\R=\R\cup\{\pm\infty\}$ 
is measurable along time-ordered a.c.\ curves if for every 
time-ordered a.c.\ curve $\sfq$ in $\QQ$ the composition $G\circ\sfq$
is Lebesgue measurable.
We denote by $\rmM(\QQ)$ the collection of all such functions.

An \emph{evolution system} $(\XX,\sfd,\cE,\cF ,\cP)$ consists of 
\begin{enumerate}
\item a complete extended metric space $(\XX,\sfd)$, 
\item an energy
  functional $\cE:\QQ\to \R\cup\{+\infty\}$ in $\rmM(\QQ) $,
\item a slope functional
  $\cF :\QQ\to [0,\infty]$ in $\rmM(\QQ)$,
\item a power functional $\cP:\QQ\times[0,\infty]\to
  \overline \R$ such that for every $(q,f)\in \QQ\times[0,\infty)$
  the map
  $\cP(\cdot, f)$ belongs to $\rmM(\QQ)$ %is measurable along time-ordered a.c.\ 
  and the map
  $\cP(q,\cdot)$ is nondecreasing and upper semicontinuous.
\end{enumerate}
Notice that if $\sfq=(\sft,\sfu):[\alpha,\beta]\to\QQ$ is a time-ordered a.c.\
curve and $\sF\in \rmM_+([\alpha,\beta])$, the composition
$s\mapsto \cP(\sfq(s),\sF(s))$ is measurable.

% $\cF $ will often stand for its metric slope
% \begin{equation}\label{eq:38}
%   \cF (t,u):=\limsup_{w\to
%     u}\frac{[\cE(t,u)-\cE(t,w)]_+}{\sfd(w,u)}\quad\text{when }\cE(t,u)<\infty
% \end{equation}
% extended to $+\infty$ when $\cE(t,u)=+\infty$, or a suitable
% lower-semicontinuous relaxation, 
% and $\cP(t,u)$ will be the partial derivative of the energy with respect
% to time: $\cP(t,u):=\partial_t \cE(t,u)$, at least when $\cE(t,u)<\infty$.

The essential feature of this structure is 
captured by the following definition:
\begin{definition}[Irreversible upper gradients for time-dependent
  functionals]
  \label{def:upper-gradients}
  We
  say that $(\XX,\sfd,\cE,\cF ,\cP)$ is an (irreversible) upper
  gradient system if for every time-ordered a.c.\ curve
  $[\alpha,\beta]\ni s\mapsto \sfq(s)=(\sft(s),\sfu(s))\in \QQ$ 
  and every $\sF\in \rmM_+([\alpha,\beta])$ satisfying
  \begin{displaymath}
    \cE({\sft(\alpha)},\sfu(\alpha))<\infty,\quad 
    \sF\ge \cF\circ \sfq\quad\text{in }[\alpha,\beta],\quad
    \int_\alpha^\beta \big[\cP(\sfq(s),\sF(s))\big]_-\dot\sft(s)\,\d
    s<\infty%\circ(\sfq,\sF))_-\,\dot \sft\in L^1(\alpha,\beta),
  \end{displaymath}
  there holds
  \begin{equation}
    \label{eq:4}
    \cE(\sfq(\alpha))+
    \int_{\alpha}^{\beta} \cP(\sfq(s),\sF(s))\dot\sft(s)\,\d s\le 
    \cE({\sfq(\beta)})+\int_{\alpha}^{\beta}
    \sF(s)%\cF ({\sfq(s)})
    \,|\dot\sfu|(s)\,\d s.
  \end{equation}
\end{definition}
A \emph{metric dissipation function} is a 
function
\begin{equation}
  \label{eq:1}
  %\text{convex function}\quad
  \psi:[0,\infty)\to [0,\infty)\ \text{convex, with }\psi(0)=0,\quad
  \sfL=\lim_{v\to+\infty}\frac{\psi(v)}v>0.
\end{equation}
We say that $\psi$ has $\sfL$-linear growth if $\sfL<+\infty$ and that 
$\psi$ is \emph{superlinear} if $\sfL=+\infty$. 
Its dual $\psi^*:[0,\infty)\to [0,\infty]$ is defined as 
\begin{equation}
  \label{eq:39}
  \psi^*(f)=\sup_{v\ge 0}\Big(f\,v-\psi(v)\Big),
\end{equation}
and it is a convex and superlinear function with $\psi^*(0)=0$ as
well, satisfying the Fenchel duality
\begin{equation}
  \label{eq:40}
  \begin{gathered}
    \psi(v)+\psi^*(f)\ge f\,v\quad \forevery v,f\in [0,\infty);\\
    \psi(v)+\psi^*(f)= f\,v\quad\Leftrightarrow\quad
    f\in \partial \psi(v)\quad \Leftrightarrow\quad
    \psi'(v_-)\le f\le \psi'(v_+)
  \end{gathered}
\end{equation}
where $\partial\psi(v)=[\psi'(v_-),\psi'(v_+)]$ denotes the convex
subdifferential of $\psi$. 
Notice that at $v=0$ we have
\begin{displaymath}
  \partial\psi(0)=[0,\psi'(0_+)]
\end{displaymath}
so that $\partial\psi(0)$ is single valued only when the right
derivative of $\psi$ at $0$ vanishes.

The proper domain $D(\psi^*):=\{f\in
[0,\infty):\psi^*(f)<\infty\}$ 
is related to $\sfL$ by the relation
\begin{equation}
  \label{eq:116}
  \sfL=\sup\{f:\psi^*(f)<\infty\},
\end{equation}
so that $\psi^*$ is finite in $[0,\infty)$ if and only if $\psi$ is superlinear.
The typical examples are
\begin{displaymath}
  \psi(v)=\frac 1p v^p,\quad
  \psi^*(f)=\frac1{p^*}f^{p^*},\quad
  \partial\psi(v)=v^{p-1};\qquad
  p>1,\quad 
  \frac 1p+\frac1{p^*}=1;
\end{displaymath}
\begin{displaymath}
  \psi(v)=\sfL v,\quad \psi^*(f)=  \begin{cases}
    0&\text{if }f\le \sfL,\\
    +\infty&\text{if }f>\sfL.
  \end{cases}
\end{displaymath}
The $\psi$-gradient flows associated with an evolution system can be
characterized by a simple family of dissipation inequalities:
\begin{definition}[Energy-dissipation inequality]
  \label{def:ED}
  Let $(\XX,\sfd,\cE,\cF ,\cP)$ be an evolution system and let $\psi$ be a
  metric dissipation function. \\
  A curve $u\in \AC {}{[0,T]}{(\XX,\sfd)}$ with $\cE(0,u(0))<\infty$ 
  satisfies the $\psi$-$\psi^*$
  energy-dissipation inequality
  if there exists a measurable map $\sF\in \rmM_+([0,T])$ 
  satisfying
  \begin{equation}
    \label{eq:29}
    \sF(t)\ge \cF(t,u(t))\quad\text{in $[0,T]$},\qquad
    \int_0^T \big[\cP(t,u(t),\sF(t))\big]_+\,\d t<\infty,
  \end{equation}
  and 
  for every $t\in [0,T]$
  \begin{equation}
    \label{eq:5}
    \begin{aligned}
      \cE(t,u(t))&+\int_{0}^{t} \Big(\psi\big(|\dot
      u|_\sfd(r)\big)+\psi^*\big(\sF(r)\big)\Big) \,\,\d r \\&\le
      \cE(0,u(0))+ \int_{0}^{t} \cP(r,u(r),\sF(r))\,\d r.
    \end{aligned}
  \end{equation}
\end{definition}
It is immediate to see that if $(\XX,\sfd,\cE,\cF ,\cP)$ is an \emph{upper
gradient} evolution system, then by \eqref{eq:4} and \eqref{eq:40} 
the integral characterization \eqref{eq:5} is equivalent to 
the following properties:
\begin{subequations}
\label{eq:sub1}
  \begin{align}
    \label{eq:6}
    &t\mapsto \cE(t,u(t))\quad\text{is absolutely continuous in
        $[0,T]$,}
      \\& 
      \label{eq:6bis}
      \sF(t)\ge \cF(t,u(t))\quad\text{a.a.\ $t$ in
      }(0,T),
      \\&
    \label{eq:7}
    \begin{aligned}
      -\frac\d{\d t}\cE(t,u(t))&+\cP(t,u(t),\sF(t))=|\dot
      u|_\sfd(t)\,\sF (t)\\&=\psi\big(|\dot
      u|_\sfd(t)\big)+\psi^*\big(\sF (t)\big)
    \end{aligned}
    \qquad\forae\ t\in (0,T).
  \end{align}
\end{subequations}
Notice that \eqref{eq:7} and \eqref{eq:40} yields the velocity-slope relation
\begin{equation}
  \label{eq:73}
  \sF(t)\in \partial\psi\big(|\dot u|_\sfd(t)\big)\quad\text{for a.a.\
    $t\in (0,T),$}
\end{equation}
and, by \eqref{eq:4}, $\sF(t)$ realizes the minimal selection property
\begin{equation}
  \label{eq:138}
  |\dot u|_\sfd(t)\,\sF(t)\
  % \psi^*\big(\sF(t)\big)
  -\cP(t,u(t),\sF(t))=\min_{f \ge
  \cF(t,u(t))}
%\Big( 
|\dot u|_\sfd(t)\,f-\cP(t,u(t),f)
%  \Big)
\end{equation}
for a.a.\ $t\in (0,T)$.
In particular, \eqref{eq:138} 
yields
\begin{equation}
  \label{eq:139}
  \sF(t)=\cF(t,u(t))\quad\text{$|\dot u|_\sfd\,\Leb 1$-a.e.\ in }(0,T) \quad
  \text{when  $\cP$ is independent of $\sF$},
\end{equation}
and \eqref{eq:139} holds $\Leb 1$-a.e.\ when $\psi'(0+)=0$.
\begin{definition}[$\psi$-gradient flows]
  \label{def:GF}
  Let $(\XX,\sfd,\cE,\cF ,\cP)$ be an upper
  gradient evolution system, and let $\psi$ be a metric dissipation
  function as in \eqref{eq:1}.
  A curve $u\in \AC {}{a,b}{(\XX,\sfd)}$ is a $\psi$-gradient flow of the
  system if it satisfies \eqref{eq:29} and \eqref{eq:5} at $t=T$, or,
  equivalently, (\ref{eq:sub1}a,b,c).
  %and \eqref{eq:7}.
\end{definition}

%% COMMENT
\subsection{BV solutions to evolution systems}
Let us now consider the case of a dissipation potential $\psi$ with linear
growth, 
corresponding to $\sfL<\infty$ in \eqref{eq:1}. In this case,
absolutely continuous solutions to (\ref{eq:sub1}a,b,c) often do 
not exist, even in the smooth and finite-dimensional setting of
Section \ref{sec:intro} and
therefore we have to extend the previous definitions to the BV
setting.

\renewcommand{\fre}{\frf}
As before, we fix the time interval $[0,T]$
and we denote by $\QQ$ the product space $[0,T]\times \XX$
and we consider a function $\fre:\QQ\to[0,\infty]$ 
measurable along absolutely continuous curves. Relevant examples will
be
$\fre:=\cF $ and 
\begin{equation}
  \label{eq:42}
  \frf(q):=\cF (q)\lor \sfL\quad\forevery q\in \QQ.
\end{equation}
%$\fre:=\cF \lor 1$.
% For a parameter $\finp\ge0$ (the most important cases we will
% consider are $\finp=0$ and $\finp=1$), we introduce 
% the ``Finsler'' functions 
% \begin{equation}
%   \label{eq:15}
%   \frfa:\QQ\times [0,\infty)\to [0,\infty),\quad
%   \frfa(t,\theta,v):=(\cF _t(\theta)\lor \finp)\,v =\finp \,v+v(\cF_t(\theta)-\finp)_+;
% \end{equation}
% interpreting $v$ as the metric velocity of a given curve
% $\vartheta\in \AC{}{r_0,r_1}{(\XX,\sfd)}$, 
We interpret $\fre(t,\cdot)$ as a conformal factor that induces a modified geometry in $\XX$:
the corresponding length of a curve $\vartheta\in \AC{}{r_0,r_1}{(\XX,\sfd)}$ (notice that the curve is parametrized by a
different variable $r$, and $t$ remains fixed) is 
\begin{equation}
  \label{eq:17}
  \finsl {\sfd,\fre} t\vartheta:=\int_{r_0}^{r_1} \fre(t,\vartheta(r))\,|\dot\vartheta(r)|_\sfd\,\d r
\end{equation}
and the cost of a transition from $u_0$ to $u_1$ in $\XX$ at the time
$t\in [0,T]$ is then defined by
\begin{displaymath}
  \bicost {\sfd,\fre} t{u_0,u_1}=\inf\Big\{\finsl {\sfd,\fre} t\vartheta:\vartheta\in
  \AC{}{r_0,r_1}{(\XX,\sfd)},\ \vartheta(r_i)=u_i\Big\}.
\end{displaymath}
We also set
  \begin{align*}
    \tricost {\sfd,\fre} t{u_0,u,u_1}:&=\bicost {\sfd,\fre} t{u_0,u}+\bicost{\sfd,\fre} t{u,u_1}\\&
    \begin{aligned}
      =\inf\Big\{&\finsl{\sfd,\fre} t\vartheta:\vartheta\in
      \AC{}{r_0,r_1}{(\XX,\sfd)},\\& \vartheta(r_i)=u_i,\
      \vartheta(r)=u\text{ for some }r\in [r_0,r_1]\Big\}.
    \end{aligned}
  \end{align*}
We can thus consider a modified Jump functional
% \begin{equation}
%   \label{eq:18}
%   \pVar_{\sfd,\fre}(u;[\alpha,\beta])=\int_\alpha^\beta \d\,\comd u\sfd+
%   \pJmp{\sfd,\fre}(u;[\alpha,\beta])
% \end{equation}
% where
\begin{align*}
  \pJmp{\sfd,\fre}(u;[\alpha,\beta])&=
  \bicost{\sfd,\fre} \alpha{u(\alpha),u(\alpha_+)}\\&+\sum_{t\in
    \rmJ_u}\tricost{\sfd,\fre} t{u(t_-),u(t),u(t_+)}+
  \bicost{\sfd,\fre} \beta{u(\beta_-),u(\beta)}.
\end{align*}
The previous quantities will be quite useful to extend the chain-rule
inequality \eqref{eq:4} to
the BV setting. 
Notice that we are assuming that $\sF$ is a Borel map (instead of Lebesgue
measurable as in Definition \ref{def:upper-gradients}), since an integration with respect to the possibly
singular measure $\comd u\sfd$ occurs in \eqref{eq:46}.
\begin{proposition}
  \label{prop:BV-chain-rule}
  Let $(\XX,\sfd,\cE,\cF ,\cP)$ be an upper gradient evolution system,
  $\frf:=\cF \lor \sfL$ for some $\sfL>0$, and
  let
  $u\in \BV{0,T}{(\XX,\sfd)}$ satisfy
  \begin{displaymath}
    \cE(0,u(0))<\infty,\quad \int_0^T (\cP(t,u(t),\sF(t)))_-\, \d
    t<\infty,\quad
    \pJmp{\sfd,\frf}(u;[0,T])<\infty,
  \end{displaymath}
  for some Borel map $\sF\in \rmB_+([0,T])$ %\to \R$ 
  with $\sF(t)\ge
  \cF(t,u(t))$ in $[0,T]$.
  Then for every $t\in [0,T]$ 
  \begin{equation}
    \label{eq:46}
    \begin{aligned}
      \cE(0,u(0))&+\int_0^t \cP(r,u(r),\sF(r))\,\d r\\&\le
      \cE(t,u(t))+\int_0^t \sF(r)\,\d \comd u\sfd +\pJmp{\sfd,\cF
      }(u;[0,t]).
    \end{aligned}
  \end{equation}
\end{proposition}
\begin{proof}
  It is not restrictive to assume $t=T$. Let us denote by $(t_n)_n$ the jump set
  $\rmJ_u$ of $u$ and let us first set
  \begin{displaymath}
    \begin{gathered}
      \sfs(t):=t+\rmV_u(t),\quad \sfS:=T+\rmV_u(T),\quad
      I_n:=(\sfs(t_{n\,-}),\sfs(t_{n\,+})),\quad I:=\cup_n I_n,\\
      D:=[0,\sfS]\setminus I,\quad \sft:=\sfs^{-1}: D \to [0,T],\quad
      \sfu:=u\circ\sft:D\to \XX.
    \end{gathered}
  \end{displaymath}
  Since $ \pJmp{\sfd,\frf}(u;[0,T])<\infty $ it is not difficult to check that $\sft,\sfu$ are Lipschitz maps
  (if we only know $\pJmp{\sfd,\cF}(u;[0,T])<\infty$, it would not
  clear how to derive a uniform upper bound on the total variation of
  the function $\sfu$).
  We easily extend $\sft$ to $[0,\sfS]$ by setting
  \begin{equation}
    \label{eq:49}
    \sft(s)\equiv t_n\quad\text{if }s\in I_n.
  \end{equation}
  In order to extend $\sfu$, we fix $\eps>0$ and for every interval
  $I_n$ we consider two curves
  $\vartheta_n,\zeta_n:[\sfs(t_{n\,-}),\sfs(t_{n\,+})]\to \XX$
  satisfying
  $\vartheta_n(\sfs(t_{n\,\pm}))=\zeta_n(\sfs(t_{n\,\pm}))=u(t_{n\,\pm})$,
  taking the value $u(t_n)$ at some point in $I_n$, and fulfilling
  \begin{equation}
    \label{eq:50}
    \begin{aligned}
      \int_{I_n}\cF ({t_n},\vartheta_n(s))|\dot\vartheta_n|_\sfd(s)\,\d
      s&\le \tricost{\sfd,\cF }{t_n}{u(t_{n\,-}),u(t_n),u(t_{n\,+})}+\eps
      2^{-n},\\
      \int_{I_n}\frf({t_n},\zeta_n(s))|\dot\zeta_n|_\sfd(s)\,\d
      s&\le \tricost{\sfd,\frf}{t_n}{u(t_{n\,-}),u(t_n),u(t_{n\,+})}+\eps
      2^{-n}.
    \end{aligned}
  \end{equation}
  For $N\in \N$ we define
  \begin{align*}
    \sfu_N(s)&:=
    \begin{cases}
      u(s)&\text{if }s\in [0,\sfS]\setminus I,\\
      \vartheta_n(s)&\text{if }s\in I_n,\ n\le N,\\
      \zeta_n(s)&\text{if }s\in I_n,\ n>N,
    \end{cases}
    \\
    \sF_N(s)&:=
    \begin{cases}
      \sF(\sft(s))&\text{if }s\in [0,\sfS]\setminus I,\\
      \cF(t_n,\vartheta_n(s))&\text{if }s\in I_n,\ n\le N,\\
      \frf(t_n,\zeta_n(s))&\text{if }s\in I_n,\ n>N.
    \end{cases}
  \end{align*}
  It is not difficult to check that $\sfu_N$ is absolutely continuous,
  so that \eqref{eq:4} yields (see \cite[Lemma 4.1]{Rossi-Mielke-Savare08})
  \begin{align*}
    \cE&(0,u(0))+\int_0^T \cP(t,u(t),\sF(t))\,\d t\\&=
    \cE(\sft(0),\sfu_N(0))+\int_0^\sfS
    \cP(\sft(s),\sfu(s),\sF(\sft(s)))\,\dot\sft(s)\,\d s\\&\le
    \cE(\sft(\sfS),\sfu_N(\sfS))+\int_0^\sfS
    \sF_N (s)|\dot\sfu_N|_\sfd(s)\,\d s\\&=
    \cE(T,u(T))+\int_D  \sF (\sft(s))%,u(\sft(s)))
    |\dot\sfu_N|_\sfd(s)\,\d
    s\\&\ +
    \sum_{n=1}^N
    \int_{I_n}\cF (t_n,\vartheta_n(s))|\dot\vartheta_n|_\sfd(s)\,\d s
    +    \sum_{n>N}
    \int_{I_n}\frf(t_n,\zeta_n(s))|\dot\zeta_n|_\sfd(s)\,\d s\\&\le
     \cE(T,u(T))+\int_0^T \sF (t)%,u(t))
     \,\d |u'|+\eps\\&\ +
     \sum_{n=1}^N
     \tricost{\sfd,\cF }{t_n}{u(t_{n\,-}),u(t_n),u(t_{n\,+})}
     +\sum_{n>N} \tricost{\sfd,\frf}{t_n}{u(t_{n\,-}),u(t_n),u(t_{n\,+})}.
  \end{align*}
  Passing first to the limit as $N\up\infty$ (notice that 
  the last term vanishes as $N\up\infty$ since 
  $\pJmp{\sfd,\frf}(u;[0,T])$ is finite) and then as $\eps\down0$
  we obtain \eqref{eq:46}.
\end{proof}

  % In particular, if
  % \begin{equation}
  %   \label{eq:47}
  %   \cF (t,u(t))\le 1\quad \text{for $\comd u\sfd$-a.e.\ $t\in [0,T]$},
  % \end{equation}
  % then
  % \begin{equation}
  %   \label{eq:45}
  %   \cE(0,u(0))+\int_0^t \cP(r,u(r))\,\d r\le
  %   \cE(r,u(r))+\pVar_{\sfd,\cF }(u;[0,t])\quad
  %   \forevery t\in [0,T].
  % \end{equation}
\begin{definition}[Energy-dissipation inequality for $\mathrm{BV}$
  functions]
  \label{def:EDBV}
  Let\\
  $(\XX,\sfd,\cE,\cF ,\cP)$ be an evolution 
  system in the time interval $[0,T]$, 
  let $\psi$ be a metric dissipation function with $\sfL$-linear growth, and
  let $\frf:=\cF \lor \sfL$.
  A curve $u\in \BV {0,T}{(\XX,\sfd)}$ with $\cE(0,u(0))<\infty$ 
  satisfies the $\psi$-$\psi^*$
  energy dissipation inequality if 
  there exists a Borel map
  $\sF\in \rmM_+([0,T])$ %\to [0,\infty]$ 
  satisfying \eqref{eq:29},
  \begin{equation}
    \label{eq:30}
    \sF(t)\ge \cF(t,u(t))\quad\text{in }[0,T],\quad    
    \int_0^T \big[\cP(t,u(t),\sF(t))\big]_+\,\d t<\infty,
  \end{equation}
   the \emph{stability condition on the Cantor part}
  \begin{equation}
    \label{eq:21}
    \sF (t)\le \sfL\quad\text{for $\Cmd u\sfd$-a.a.\ $t\in [0,T]$}
  \end{equation}
  and
  \begin{align}
  %  \begin{aligned}
      \cE(t,u(t))&+\int_{0}^{t} \Big(\psi\big(|\dot
      u|_\sfd(r)\big)+\psi^*\big(\sF (r)%,u(r))
      \big)\Big) \,\,\d r
      +\sfL\int_0^t \d\Cmd u\sfd+\pJmp{\sfd,\frf}(u;[0,t])
      \notag\\&\le
      \cE(0,u(0))+ \int_{0}^{t} \cP(r,u(r),\sF(r))\,\d r
      \quad\forevery t\in [0,T].
  \label{eq:5bv}            \tag{ED}
%    \end{aligned}
  \end{align}
\end{definition}
Since $\psi^*(f)=\infty$ if $f>\sfL$, \eqref{eq:5bv} yields in fact a
stronger version of the local stability condition
\begin{equation}
  \label{eq:118}
  \sF (t)\le \sfL\quad\text{for $(\Leb1+\Cmd u\sfd)$-a.a.\ $t\in [0,T].$}
  \tag{$\mathrm S_{\rm loc}$}
\end{equation}
In the rate-independent case $\psi(v)=\sfL v$, \eqref{eq:21} and
\eqref{eq:5bv} are thus equivalent to 
\eqref{eq:118} and 
 \begin{equation}
   \label{eq:5bvbis}
   \tag{EDRI}
   \begin{aligned}
     \cE(t,u(t))&
     +\sfL\int_{0}^{t} \d\,\comd u\sfd
     +\pJmp{\sfd,\frf}(u;[0,t])
     \\&\le
     \cE(0,u(0))+ \int_{0}^{t} \cP(r,u(r),\sF(r))\,\d r
  \end{aligned}
 \end{equation}
for every $t\in [0,T]$.
\begin{definition}[$\mathrm{BV}$ solutions to evolution systems and
  rate-independent flows]
  \label{def:BVsol}
  Let $(\XX,\sfd,\cE,\cF ,\cP)$ be an upper
  gradient system in the time interval $[0,T]$, 
  let $\psi$ be a metric dissipation function with $\sfL$-linear growth, and
  let $\frf:=\cF \lor \sfL$.\\
  A curve $u\in \BV {0,T}{(\XX,\sfd)}$ with $\cE(0,u(0))<\infty$ is a $\mathrm{BV}$ \emph{solution} of
  the corresponding evolution if there exists %a Borel map
  $\sF\in \rmB_+([0,T])$ %\to [0,\infty]$ 
  satisfying \eqref{eq:30}, 
  the \emph{local stability condition} \eqref{eq:118} and
  the \emph{energy balance}
  \begin{align}
%     \begin{aligned}
    \notag\cE(t_2,u(t_2))&+\int_{t_1}^{t_2} \Big(\psi\big(|\dot
      u|_\sfd\big)+\psi^*\big(\sF %,u(r))
      \big)\Big) \,\,\d r
      +\sfL\int_{t_1}^{t_2} \d\Cmd u\sfd+\pJmp{\sfd,\frf}(u;[t_1,t_2])\\&=
      \cE(t_1,u(t_1))+ \int_{t_1}^{t_2} \cP(r,u(r),\sF(r))\,\d
      r
    \label{eq:22}
       \tag{EB}
      % \qquad\forevery 
%    \end{aligned}
  \end{align}
  for every $[t_1,t_2]\subset [0,T].$
\end{definition}  
Notice that
\eqref{eq:22} holds if and only if 
the curve $t\mapsto e(t):=\cE(t,u(t))$
(extended to $\R$ as in \eqref{eq:140})
is of bounded variation,
$\rmJ_e= \rmJ_u$,
and its distributional time derivative $\frac \d{\dt} e$ satisfies 
  \begin{equation}
    \label{eq:21bis}
    -\frac\d{\d t} e+\cP(\cdot,u,\sF)= \Big(\psi\big(|\dot
    u|_\sfd\big)+\psi^*\big(\sF \big)\Big)\Leb 1+\sfL\Cmd
    u\sfd-\rmJ_e\quad\text{in }\R,
  \end{equation}
  where at each jump point $t\in \rmJ_u$ we have
  $e(t_\pm)=\cE(t,u(t_\pm))$ and the jump part $\rmJ_e$ is
  \begin{displaymath}
    \left\{\begin{aligned}
        \cE(t,u(t_-))-\cE(t,u(t_+))&=-\rmJ_e(\{t\}),\\
        \cE(t,u(t_-))-\cE(t,u(t))&=\bicost \sfd\frf{u(t_-),u(t)},\\
      \cE(t,u(t))-\cE(t,u(t_+))&=\bicost \sfd\frf{u(t),u(t_+)}.
    \end{aligned}
    \right.
  \end{displaymath}
As for gradient flows, thanks to Proposition \ref{prop:BV-chain-rule}
it is immediate to see that whenever \eqref{eq:21} holds
the energy balance \eqref{eq:22} is equivalent to the energy-dissipation inequality
\eqref{eq:5bv} at the final point $t=T$.
Moreover, a BV solution $u$ satisfies
\begin{align}
  \label{eq:141}
  \sF(t)&\in \partial\psi\big(|\dot u|_\sfd(t)\big) &&\text{for
    $\Leb 1$-a.a.\ $t\in (0,T)$},\\
  \sF(t)&=\cF(t,u(t))=\sfL&&\text{for $\Cmd  u\sfd$-a.a.\ $t\in
    (0,T)$},
\end{align}
and the minimal selection principle
\begin{equation}
  \label{eq:138bis}
  |\dot u|_{\sfd}(t) \,\sF(t) 
  % \psi^*\big(\sF(t)\big)
  -\cP(t,u(t),\sF(t))=\min_{f \ge
    \cF(t,u(t))}
  %\Big\{: 
  |\dot u|_{\sfd}(t)\,f 
  %\psi^*\big(f\big)
  -\cP(t,u(t),f)
  %\Big\}
\end{equation}
for $\Leb 1$-a.a.\ $t\in (0,T)$.
In the rate-independent case $\psi(v)=\sfL v$, a BV solution is
equivalently characterized by the local stability \eqref{eq:118} and
the energy balance
 \begin{equation}
   \tag{EBRI}
   \begin{aligned}
     \cE(t_2,u(t_2))&+\sfL\int_{t_1}^{t_2} \d\,\comd u\sfd
     +\pJmp{\sfd,\frf}(u;[t_1,t_2])
     \\&= \cE(t_1,u(t_1))+ \int_{t_1}^{t_2}
     \cP(r,u(r),\sF(r))\,\d r
   \end{aligned}
   \end{equation}
for every $0\le t_1\le t_2\le T$.
\section{Compactness and convergence for families of Gradient Flows}
\label{sec:main}
In this section we will state and prove our main results.
For the sake of clarity, we distinguish between the compactness
(Theorem \ref{thm:compactness}) and
the stability (Theorem \ref{thm:convergence}) issues.

We also take care to highlight the role of the energy-dissipation
inequality for general metric-evolution systems
$(\XX,\sfd_h,\cE_h,\cF_h,\cP_h)$, even if they 
do not satisfy the irreversible upper gradient condition
\ref{def:upper-gradients}. Therefore, compactness and stability of the
energy-dissipation inequality always hold whenever suitable
topological
properties (see the next (C1,2,3,4) assumptions) are satisfied. 

In order to recover a $\psi$-gradient flow
or a BV solution in the limit, we will ask that 
$(\XX,\sfd,\cE,\cF,\cP)$ is an upper gradient system.

We also notice that our theorems can also be extremely useful to prove existence
results for solutions to the limit evolution system: in this case one
could think 
that $(\XX,\sfd_h,\cE_h,\cF_h,\cP_h)$ is a family of suitably
regularized problems
(e.g.\ with smooth superlinear dissipations and better energies) for
which
existence is already known (see \cite{Mielke-Rossi-Savare12}).

Let $(\XX,\top)$ be a topological space, and let $\QQ=[0,T]\times\XX$ with the
standard product topology, which we will denote by $\pi$.
We consider a family of evolution systems $(\XX,\sfd_h,\cE_h,\cF_h,\cP_h)$ in the time interval $[0,T]$ indexed
by the parameter $h\in\N$, a sequence $\bar u_h$ of initial points, and
metric dissipation functions
$\psi_h,\psi$ such that 
\begin{equation}
  \label{eq:53}
  \lim_{h\to\infty}\psi_h(v)=\psi(v)\quad\forevery v\in
  [0,\infty),
\end{equation}
Since each function $\psi_h$ is monotone, \eqref{eq:53} is equivalent to 
\begin{equation}
  \label{eq:60}
    \Gamma\text{-}\lim_{h\to\infty}\psi_h(v)=\psi(v)\quad\forevery v\in [0,\infty),
\end{equation}
and also to the following property, valid for
arbitrary sequences $(w_h)_h\subset [0,\infty)$:
\begin{equation}
  \label{eq:54}
  w_h\to w\quad
  \Rightarrow\quad
  \liminf_{h\to\infty}\psi_h(w_h)\ge \psi(w),\quad
  \liminf_{h\to\infty}\psi_h^*(w_h)\ge \psi^*(w).
  % \begin{cases}
  %   0&\text{if }w\le 1,\\
  %   +\infty&\text{if }w>1.
  % \end{cases}
\end{equation}
Typical examples are given in \eqref{eq:55} and \eqref{eq:87}.
% \begin{equation}
%   \label{eq:55}
%   \begin{aligned}
%     \psi_h(v)&=\frac 1{p_h}\,v^{p_h}\quad&&\text{with }p_h>1, \quad
%     &&\lim_{h\to\infty}p_h=1\\
%     \psi_h(v)&=v+ {\eps_h}\,v^p\quad&&\text{with }p>1,\ \eps_h>0,\quad &&\lim_{h\to\infty}\eps_h=0.
%   \end{aligned}
% \end{equation}
We want to study the limit of absolutely continuous $\psi_h$-gradient
flows $u_h\in \AC{}{0,T}{(\XX,\sfd_h)}$ of the systems
$(\XX,\sfd_h,\cE_h,\cF_h,\cP_h)$ with $u_h(0)=\bar u_h$ as $h\to\infty$,
assuming that they ``converge'' (in a variational sense that we are
going to make precise) to a limit
system $(\XX,\sfd,\cE,\cF ,\cP)$.
The most interesting case is when $\psi$ has $\sfL$-linear growth, so
that we expect a function of bounded variation in the limit.

Here and in the following we identify
diverging subsequences in $\N$ with subsets $H\subset \N$ with $\sup
H=\infty$ and we write $\lim_{\hinH}$
for $\lim_{h\to\infty, h\in H}$.

We will assume that:
\begin{itemize}
\item[(C1)] There exist constants $\sfa<\sfL,\sfb\ge 0$ such that 
  \begin{equation}
    \label{eq:63}
    \tilde\cE_h(t,u):=\cE_h(t,u)+\sfa \sfd_h(u,\bar u_h)+\sfb
    \ge0
    \quad\forevery (t,u)\in \QQ.
  \end{equation}
\item[(C2)] The energies $(\tilde\cE_h)_{h\in \N}$ are equi-coercive in
  $\QQ$:
  for every sequence $(q_h)_{h\in H}\subset \QQ$ with $\sup_{h\in H}
  \tilde\cE_h(q_h)<\infty$ there exists a subsequence $H'\subset H$ such
  that $\lim_{h\in H'}q_h=q$ in the $\pi$-topology.
  % \begin{equation}
  %   \label{eq:23}
  %   \big\{u\in \XX:\cE_t(u)\le c\big\}\quad\text{are compact
  %     in $(\XX,\sigma)$ for every $t\in [0,T]$ and $c\in \R$}
  % \end{equation}
\item[(C3)]
  If two sequences $q^i_h=(t^i_h,x^i_h)\subset \QQ$, $h\in H$, $i=1,2$, satisfy
  $\sup_{h\in H}
  \tilde\cE_h(q^i_h)<\infty$ 
  and $\pi$-converge to $q^i$, 
  then we
  have
  \begin{equation}
    \label{eq:25}
    \liminf_{\hinH}\sfd_h(x^1_h,x^2_h)\ge\sfd(x^1,x^2)
  \end{equation}
  where $\sfd(\cdot,\cdot)$ is a complete extended distance on $\XX$.
\end{itemize}

\begin{theorem}[A priori bounds and compactness]
  \label{thm:compactness}
  Let us suppose that {\upshape (C1)} holds and let $u_h\in
  \AC{}{0,T}{(\XX,\sfd_h)}$,
  $\sF_h\in \rmB_+([0,T])$ be sequences
  satisfying $u_h(0)=\bar u_h$ and the corresponding energy-dissipation inequalities
  \eqref{eq:29} and \eqref{eq:5} for every $h\in \N$, namely
  \begin{equation}
    \label{eq:29bis}
    \sF_h(t)\ge \cF_h(t,u_h(t))\quad\text{in $[0,T]$},\qquad
    \int_0^T \big[\cP_h(t,u_h(t),\sF_h(t))\big]_+\,\d t<\infty,
  \end{equation}
  and
\begin{equation}
  \label{eq:5bis}
  \begin{aligned}
    \cE_h({t},u_h(t))&+\int_{0}^{t} \Big(\psi_h\big(|\dot
    u_h|_{\sfd_h}(r)\big)+\psi^*_h\big(\sF_h(r)%,u_h(r))
    \big)\Big) \,\,\d r
    \\&\le \cE_h({0},\bar u_h)+ \int_{0}^{t}
    \cP_h({r},u_h(r),\sF_h(r))\,\d r
  \end{aligned}
\end{equation} 
for every $t\in [0,T]$.
If there exists a constant $A\ge 0$ such that
  \begin{equation}
    \label{eq:24}
    \cE_h(0,\bar u_h)+
    \int_{0}^{t} \cP_h\big({r},u_h(r),\sF_h(r)\big)\,\d r\le A\quad\forevery h\in
    \N,\ t\in [0,T],
  \end{equation}
  then there exists a constant $C>0$ such that 
  for every $t\in [0,T]$ and $h\in \N$
  \begin{equation}
    \label{eq:67}
    \cE_h(t,u_h(t))\le C,\quad
    \int_0^T \Big(\psi_h\big(|\dot u_h|_{\sfd_h}(r)\big)\,\d
    r+
    \psi^*_h\big(\sF_h(r)%,u(r))
    \big)\Big)
    \,\,\d r\le C,
  \end{equation}
  so that 
  \begin{equation}
    \label{eq:56}
    \lim_{h\to\infty}
    \Leb 1\{t\in (0,T):\sF_h(t)%,u_h(t))
    \ge
    f\}=0\quad\forevery f>\sfL.
  \end{equation}
  If moreover {\rm (C2,3)} hold, then for every subsequence $H\subset \N$ there exists a further subsequence
  $H'\subset H$ 
  such that 
  \begin{equation}
    \lim_{h\in H'}u_h(t)=u(t) \quad
    \text{in
      $(\XX,\sigma)$ for every $t\in [0,T]$,}
    \label{eq:68}
\end{equation}
      with $u\in \BV{0,T}{(\XX,\sfd)}$,
  \begin{equation}
    \label{eq:57}
    \lim_{h\in H'}u_h(t_h)=u(t)\quad
    \text{if $[0,T]\ni t_h\to t\in 
      [0,T]\setminus\rmJ_u$},
  \end{equation}
  and, for every interval $[t_1,t_2]\subset [0,T]$,
  \begin{equation}
    \label{eq:120}
    \liminf_{h\in H'}\int_{t_1}^{t_2}\psi_h\big(|\dot
    u_h|_{\sfd_h}(r)\big)\,\d r\ge 
    \int_{t_1}^{t_2} \psi\big(|\dot
    u|_{\sfd}(r)\big)\,\d r+\sfL \Cmd u\sfd([\alpha,\beta])
  \end{equation}
\end{theorem}
\begin{proof}
  \underline{Let us first prove \eqref{eq:67} and \eqref{eq:56}.}\\
    First of all we show that $\sfd_h(u_h(t),\bar u_h)$ is uniformly
  bounded, 
  we choose $\bar \sfa \in (\sfa, \sfL)$, and we observe that
  \eqref{eq:54}, the monotonicity of 
  $\psi_h$, and the continuity of $\psi^*$ in $[0,\sfL)$ yield $\lim_{h\to\infty}\psi_h^*(\bar \sfa)=\psi^*(\bar\sfa)<\infty$ so that 
  $\sfc:=\sup_h\psi^*_h(\bar \sfa)<\infty$. It follows that 
  \begin{displaymath}
    \psi_h(v)\ge \bar \sfa\, v-\sfc\quad\forevery v\ge0,\ h\in \N,
  \end{displaymath}
  and therefore
  \eqref{eq:5bis} and \eqref{eq:63} yield
  \begin{align*}
    \bar \sfa\,\sfd_h(u_h(t),\bar u_h)&\le 
    \bar \sfa\,\int_0^t |\dot u_h|_{\sfd_h}(r)\,\d r\le 
    \int_0^t \psi_h\big(|\dot u_h|_{\sfd_h}(r)\big)\,\d r+\sfc\\&\le
    A-\cE_h(t,u_h(t))+\sfc\le 
    A+\sfb+\sfa\sfd_h(u_h(t),\bar u_h)+\sfc,
  \end{align*}
  so that 
  \begin{equation}
    \label{eq:66}
    \sfd_h(u_h(t),\bar u_h)\le (\bar
    \sfa-\sfa)^{-1}\big(A+\sfb+\sfc\big)\quad
    \forevery t\in [0,T],\ h\in \N.
  \end{equation}
  Combining (C1), \eqref{eq:24} and \eqref{eq:66} we conclude that there exists a constant $B\ge 0$ such that
  \begin{equation}
    \label{eq:64}
    -B\le \cE_h({t},u_h(t))\le A\quad \forevery  t\in [0,T],\ h\in \N.
  \end{equation}
  % A similar argument shows that (here $\displaystyle\media_{t_1}^{t_2}=(t_2-t_1)^{-1}\int_{t_1}^{t_2}$)
  % \begin{equation}
  %   \label{eq:65}
  %   \psi\Big(\limsup_{h\to\infty}\media_{t_1}^{t_2} |\dot u_h|_{\sfd_h}(t)\,\d t\Big)
  %   \le 
  %   \limsup_{h\to\infty} \media_{t_0}^{t_1} \psi_h\big( |\dot
  %   u_h|_{\sfd_h}(t)\big)\,\d t\quad
  %   \forevery 0\le t_0< t_1\le T.
  % \end{equation}
  % It is sufficient to notice that for every $0<f<\sfL$
  % \begin{align*}
  %   f\, \limsup_{h\to\infty}\media_{t_1}^{t_2} |\dot u_h|_{\sfd_h}(t)\,\d t&
  %   \le 
  %   \limsup_{h\to\infty}\media_{t_1}^{t_2} \Big(\psi_h\big(|\dot u_h|_{\sfd_h}(t)\big)+\psi_h^*(f)\Big)\,\d t\\&=
  %   \limsup_{h\to\infty} \media_{t_1}^{t_2}\psi_h\big( |\dot
  %   u_h|_{\sfd_h}(t)\big)\,\d t+\psi^*(f)
  % \end{align*}
  % by \eqref{eq:54} and to apply the duality formula
  % $\psi(v)=\sup_{0<f<\sfL} fv-\psi^*(f)$.
%  \eqref{eq:65}, 
  Therefore \eqref{eq:64} and \eqref{eq:5bis} yield
  \begin{equation}
    \label{eq:35}
    \int_0^T \psi_h\big(|\dot u_h|_{\sfd_h}(t)\big)\,\d
    t\le A+B,\quad
    \int_0^T \psi_h^*\big(\sF_h(t)%,u_h(t))
    \big)\,\d t\le A+B.
  \end{equation}
  We eventually obtain \eqref{eq:56}, since by the
  monotonicity of $\psi_h^*$ we get from the second of \eqref{eq:35}
  \begin{displaymath}
    \psi^*_h(f) \, \Leb 1\{t\in (0,T):\cF_h(t,u_h(t))\ge
    f\}\le A+B\quad\forevery f\ge 0,
  \end{displaymath}
  and $\lim_{h\to\infty}\psi^*_h(f)=\infty$ when $f>\sfL$ by
  \eqref{eq:54} and \eqref{eq:116}.

  \underline{The proof of \eqref{eq:68} and \eqref{eq:57}} can be
  easily obtained by adapting the 
  argument of the extended Ascoli-Arzel\`a-Helly type result 
  \cite[Proposition 3.3.1]{Ambrosio-Gigli-Savare08}. 

  By (C2) and the bound \eqref{eq:67}, for every $t\in [0,T]$ the
  sequence $(u_h(t))_{h\in H}$
  admits a $\sigma$-converging subsequence (possibly depending on
  $t$). 
  For every $f\in [0,\sfL)$ and $0\le t_0<t_1\le T$ we recall the bound
  \begin{equation}
    \label{eq:115}
    f\sfd(u_h(t_1),u_h(t_0))\le \int_{t_0}^{t_1} f|\dot
    u_h|_{\sfd_h}\,\d t\le  \int_{t_0}^{t_1}\Big(\psi_h\big(|\dot
    u_h|_{\sfd_h}\big)+\psi_h^*(f)\Big)\,\d t.
  \end{equation}
  We consider the nonnegative finite measures on $[0,T]$
  \begin{equation}
    \label{eq:72}
    \nu_{h,f}:=\Big(\psi_h\big(|\dot
    u_h|_{\sfd_h}\big)+\psi_h^*(f)\Big)\Leb 1,\qquad
    f\in [0,\sfL)
  \end{equation}
   on 
   $[0,T]$; up to extracting a suitable subsequence, we can suppose
   that they weakly$^*$ converge to a finite measure
   $\nu_f=\nu_0+\psi^*(f)\Leb 1$ in the
   duality with continuous functions on $[0,T]$, so that
  \begin{equation}
    \label{eq:36}
    f\limsup_{h\to\infty}\sfd_h(u_h(t_1),u_h(t_2))\le
    \limsup_{h\to\infty}\nu_{h,f}([t_1,t_2])\le \nu_f([t_1,t_2])
  \end{equation}
  for every $0\le
    t_1\le t_2\le T$.
  Denoting by $J:=\{t\in
  [0,T]:\nu_0(\{t\})>0\}$
  and considering a countable set $I\supset J$ dense in $[0,T]$, by a
  standard diagonal argument we can find a subsequence $H'\subset H$
  such that $u_h(t)\topto u(t)$ for every $t\in I$ as $h\to\infty$,
  $h\in H'$. By (C3) we have
  \begin{equation}
    \label{eq:37}
    f\,\sfd(u(t_1),u(t_2))\le \nu_f([t_1,t_2])\quad\forevery t_1,t_2\in I.
  \end{equation}
  Since $(\XX,\sfd)$ is complete, the curve $I\ni t\mapsto u(t)$ can be uniquely extended to a
  continuous curve in $[0,T]\setminus J$, which we still denote by
  $u$. 
  In order to prove \eqref{eq:57} we argue by contradiction and we
  find a sequence $H''\subset H'$, points $t_h\to t\in [0,T]\setminus J$
  and a $\sigma$-neighborhood $U$ of $u(t)$ 
  such that $(u_h(t_h))\not\in U$ for every $h\in H''$. 
  Up to extracting a further subsequence (still denoted by $H''$) we
  can assume that $u_h(t_h)\topto \tilde u\neq u(t)$ so that by (C3)
  \begin{displaymath}
    f\sfd(u(t),\tilde u)\le \liminf_{h\in H''}
    f\sfd_h(u_h(t),u_h(t_h))\le \limsup_{h\in H''}\nu_{h,f}([t,t_h])=\nu_f(\{t\})=0.
  \end{displaymath}
  This yields in particular that $u_h(t)$ converges pointwise to
  $u(t)$ as $h\to\infty$, $h\in H'$; \eqref{eq:37} then holds for every $t_1,t_2\in [0,T]$ and
  shows that 
  $u\in \BV{0,T}{(\XX,\sfd)}$.

  %Recalling \eqref{eq:65}, 
  Since for an arbitrary subdivision
  $t_0=\alpha<t_1<\cdots<t_{n-1}<t_n=\beta$ there holds
  \begin{equation}
    \label{eq:80}
    \sum_{i=1}^n\sfd_h(u_h(t_i),u_h(t_{i-1}))\le \int_\alpha^\beta|\dot
    u_h|_{\sfd_h}(r)\,\d r
  \end{equation}
  it is easy to check that 
  \begin{displaymath}
    f \sum_{i=1}^n\sfd(u(t_i),u(t_{i-1}))\le \limsup_{h\in H''}\int_\alpha^\beta|\dot
    u_h|_{\sfd_h}(r)\,\d r\le \nu_0([\alpha,\beta])+\psi^*(f)(\beta-\alpha),
  \end{displaymath}
  so that 
\begin{equation}
  \label{eq:81}
  f \Var_\sfd(u;[\alpha,\beta])\le \nu_0([\alpha,\beta]) +\psi^*(f)(\beta-\alpha).
\end{equation}
Therefore, the duality formula
$\psi(v)=\sup_{0<f<\sfL} \big(fv-\psi^*(f)\big)$ yields
\begin{equation}
  \label{eq:119}
  \psi\Big((\beta-\alpha)^{-1}\Var_\sfd(u;[\alpha,\beta])\Big)\le  (\beta-\alpha)^{-1}\nu_0([\alpha,\beta]) .
\end{equation}
From \eqref{eq:81} we immediately get
\begin{equation}
  \label{eq:121}
  f\big(\Cmd u\sfd+\jmd u\sfd\big)\le \nu_0
  +\psi^*(f)\Leb 1\quad \forevery f<\sfL.
\end{equation}
Since $\Cmd u\sfd$ and $\jmd u\sfd$ are concentrated in a $\Leb
1$-negligible set, we deduce 
\begin{equation}\label{eq:142}
  \sfL\big(\Cmd u\sfd+\jmd u\sfd\big)\le \nu_0.
\end{equation}
When $\psi$ is superlinear we conclude that $u$ is absolutely
continuous, since in this case $\sfL=\infty$ and \eqref{eq:142} yields
$\Cmd u=0,\jmd u=0$.  

\underline{Let us eventually prove \eqref{eq:120}.}\\
\eqref{eq:119} and the monotonicity of $\psi$ yield, for $\alpha=t$
and $\beta=t+\eps$
\begin{equation}
  \label{eq:122}
  \psi\Big(\frac 1\eps\int_t^{t+\eps} |\dot u|_\sfd(r)\,\d r\Big)\le \eps^{-1}\,\nu_0([t,t+\eps]).
\end{equation}
Integrating this inequality from $t_0$ to $t_1-\eps$ with respect to $t$
we obtain
\begin{align*}
  \int_{t_0}^{t_1-\eps}&\psi\Big(\eps^{-1}\int_t^{t+\eps} |\dot
  u|_\sfd(r)\,\d r\Big)\,\d t\le
  \eps^{-1}\int_{t_0}^{t_1-\eps}\nu_0([t,t+\eps])\,\d t\\&\le 
  \eps^{-1}(\Leb 1\times\nu_0)(\{(t,s)\in [t_0,t_1]^2:t\le s\le t+\eps\})
  \\&\le \eps^{-1}\int_{t_0}^{t_1} \Leb 1([s-\eps,s])\,\d\nu_0(s)=\nu_0([t_0,t_1])
\end{align*}
so that, passing to the limit  as $\eps\down0$ in the above
inequality and applying Fatou's
Lemma and Lebesgue's differentiation Theorem we get
\begin{equation}
  \label{eq:124}
  \int_{t_0}^{t_1}\psi\big( |\dot
  u|_\sfd(r)\big)\,\d r\,\d t\le \nu_0([t_0,t_1]).
\end{equation}
Since $t_0$ and $t_1$ are arbitrary, we conclude that $\nu_0\ge \psi\big( |\dot
  u|_\sfd\big)\Leb 1$. Since $\Leb 1$ is singular with respect to
  $\Cmd u\sfd$ and $\jmd u\sfd$ we eventually get
  \begin{equation}
    \label{eq:125}
    \nu_0\ge \psi\big( |\dot
  u|_\sfd\big)\Leb 1+\sfL\big(\Cmd u\sfd+\jmd u\sfd\big),
  \end{equation}
  which in particular yields \eqref{eq:120}, since 
  \begin{displaymath}
     \liminf_{h\in H'}\int_{t_1}^{t_2}\psi_h\big(|\dot
    u_h|_{\sfd_h}(r)\big)\,\d r=\liminf_{h\in
      H''}\nu_{h,0}\big((t_1,t_2)\big)\ge \nu_0\big((t_1,t_2)\big)
  \end{displaymath}
  and $\Leb 1$ and $\Cmd u\sfd$ are diffuse.
\end{proof}
We want to study now the properties of the limit function $u$;
we will suppose that the following lower semicontinuity properties hold:
\begin{enumerate}[(C4)]
\item If a sequence $(u_h)_{h\in H}\subset \XX$ 
  $\sigma$-converges to $u$ with $\sup_{h\in
    H}\tilde\cE(t,u_h)<\infty$ for some $t\in [0,T]$ then
  \begin{equation}
    \label{eq:61}
    \liminf_{h\in H}\cE_h(t,u_h)\ge \cE(t,u),
    \tag{C4$_E$}
  \end{equation}
  \begin{equation}
    \label{eq:26a}
    f_h\ge \cF_h(t,u_h),\
    f_h\to f\quad\Rightarrow\quad
    \left\{
      \begin{aligned}
        &f\ge \cF(t,u), \\&
        \limsup_{h\in H}\cP_h(t,u_h,f_h)\le
        \cP(t,u,f),
      \end{aligned}
      \right.
    \tag{C4$_{FP}$}
  \end{equation}
%  
% \begin{equation}
%   \label{eq:62}
%   \text{if }\sup_{h\in H}\cE(t,u_h)<\infty\text{ and }\lim_{h\in H}f_h=f
%     \quad\text{then}\quad
%     \tag{C4$_P$}
%   \end{equation}
  \begin{equation}
    \label{eq:26}
    %\text{if }
    \lim_{h\in H}t_h=t,\ \sup_{h\in
      H}\cE(t_h,u_h)<\infty\quad\Rightarrow
    %\text{then}
    \quad \liminf_{h\in H}
    \cF_h(t_h,u_h)\ge \cF(t,u).
    \tag{C4$_F$}
  \end{equation}
  \end{enumerate}
Notice that in \eqref{eq:26} we allow for an $h$-dependence of $t$ in
the $\Gamma$-$\liminf$ inequality for $\cF$, whereas $t$ is
independent of $h$ in \eqref{eq:26a}. \eqref{eq:26} will not be
required for
the convergence result in the superlinear case, see Theorem
\ref{thm:convergence2}.

The next statement is the main result of our paper: it states that
the energy-dissipation inequality is preserved in the
limit
for arbitrary evolution systems fulfilling (C1,2,3,4). When the limit
$(\XX,\sfd,\cE,\cF,\cP)$ is also an upper gradient evolution system,
then we recover a BV solution. 
\begin{theorem}[Stability of the energy-dissipation inequality and convergence]
  \label{thm:convergence}
  Let us assume that $\psi$ has $\sfL$-linear growth and for $h\in\N$
  let 
  $(\XX,\sfd_h,\cE_h,\cF_h,\cP_h)$ be a family of evolution
  systems satisfying {\upshape (C1,2,3,4)} with respect to a sequence
  $\bar u_h\in \XX$ $\sigma$-converging to $\bar u$.
  Let  $u_h\in \AC{}{0,T}{(\XX,\sfd_h)}$, $\sF_h\in \rmB_+([0,T])$, $h\in H$, be sequences
  satisfying the $\psi_h$-energy dissipation inequality
  \eqref{eq:5bis}, such that $u_h$ is
  pointwise converging to $u$ in $(\XX,\sigma)$, $u_h(0)=\bar
  u_h$, and 
  \begin{equation}
    \label{eq:24strong}
    \begin{aligned}
      &\lim_{h\in H} \cE_h(0,\bar u_h)=\cE(0,\bar u),
      \\&
      r\mapsto
      \big[\cP_h(r,u_h(r),\sF_h(r))\big]_+ \text{ are equi-integrable
        in $(0,T)$}.
    \end{aligned}
  \end{equation}
  %\begin{description}
  % \item[$(\sfL=\infty):$] If $\psi$ is superlinear, then $u\in \AC{}{0,T}{(\XX,\sfd)}$
  %   and satisfies the energy-dissipation inequality \eqref{eq:5}.
  % \item[$(\sfL<\infty):$] If $\psi$ has $\sfL$-linear growth, 
  Then $u\in \BV{0,T}{(\XX,\sfd)}$ satisfies 
  the local stability
  condition 
  \begin{equation}
    \label{eq:70}
    \cF(t,u(t))\le \sfL\quad\text{for every $t\in [0,T]\setminus\rmJ_u$}.
  \end{equation}
  If 
  \begin{equation}
    \label{eq:31}
    \int_0^T \cP(t,u(t),\sfL)\,\d t<\infty
  \end{equation}
  then
  $u$ satisfies the {\rm BV} energy-dissipation 
  inequality in the formulation of %\eqref{eq:30}, 
  \eqref{eq:5bv}, \eqref{eq:118}
  for some 
  $\sF\in \rmB_+([0,T])$.
  
  In particular, if
  $(\XX,\sfd,\cE,\cF,\cP)$ is an upper gradient evolution system, then
  $u$
  satisfies \eqref{eq:31} and therefore
  is a $\mathrm{BV}$-solution of the corresponding
  rate-independent evolution and we have
  \begin{gather}
    \label{eq:69}
    \lim_{h\in H}\cE_h(t,u_h(t))=\cE(t,u(t))\quad\forevery t\in [0,T].
    % \label{eq:71}
    % \begin{aligned}
    %   \lim_{h\in H}&\int_{0}^{T} \Big(\psi_h\big(|\dot
    %   u_h|_{\sfd_h}(r)\big)+\psi^*_h\big(\cF_h(r,u_h(r))\big)\Big)
    %   \,\,\d r\\&=\int_{0}^{T} \Big(\psi\big(|\dot
    %   u|_\sfd(r)\big)+\psi^*\big(\cF (r,u(r))\big)\Big) \,\,\d r
    %   +\sfL\int_0^T \d\Cmd u\sfd+\pJmp{\sfd,\frf}(u;[0,T]).
    % \end{aligned}
  \end{gather}
%\end{description}
\end{theorem}
\begin{proof}
We denote by
$\cT$ the set $[0,T]\setminus \rmJ_u$.

Let us first notice that \eqref{eq:24strong} implies \eqref{eq:24}, so
that \eqref{eq:67}, \eqref{eq:56} and \eqref{eq:57} holds. 

Fatou's Lemma yields that 
\begin{equation}
  \label{eq:32}
  \liminf_{h\in H}\sF_h(t)\le \sfL.
\end{equation}
By a diagonal argument, combining the convergence \eqref{eq:57} and
the l.s.c.\ property \eqref{eq:26} we
obtain
\begin{displaymath}
  \tilde \sF(t)= \inf \big\{\liminf_{h\in H}\sF_h(t_h):t_h\to
  t\big\}\le \sfL\quad
  \cF(t,u(t))\le \tilde \sF(t)\quad \text{for every }t\in
  \cT,
\end{displaymath}
hence we get \eqref{eq:70}.

Let us now assume \eqref{eq:31} and let us prove the BV-dissipation
inequality \eqref{eq:5bv}; it is not restrictive to prove the
inequality for $t=T$. 

We consider the function
$A:\rmI\times [0,\sfL]\to \R\cup\{\infty\}$
\begin{equation}
  \label{eq:34}
  A(t,f):=\psi^*(f)- \cP(t,u(t),f)\ge -\cP(t,u(t),\sfL).
\end{equation}
Denoting by $\cL$ the Lebesgue-measurable subsets of $\rmI$ and by
$\cB$ the Borel sets of $\R^2$, it is easy to check that $A$ is
$\cL\otimes \cB$-measurable, thanks to the monotonicity and upper
semicontinuity of $f\mapsto \cP(t,u(t),f)$.

By \eqref{eq:31}, it is the immediate to see that for a.a.\ $t\in \rmI$ 
\begin{equation}
  \label{eq:123}
  a(t):=\min\Big\{A(t,f):\cF(t,u(t))\le f\le \sfL\Big\}>-\infty
\end{equation}
Applying \cite[Lemma III.39]{Castaing-Valadier77} we find a Lebesgue
measurable map $\sF\to \R$ such that 
\begin{displaymath}
  \cF(t,u(t))\le \sF(t)\le \sfL ,\quad
  \psi^*(\sF(t))-\cP(t,u(t),\sF(t))=a(t)\qquad
  \text{for $\Leb 1$-a.a.\ $t\in \rmI$},
\end{displaymath}
and, up to a modification of $\sF$ in a negligible set, it is not
restrictive to assume $\sF$ Borel and $\sF(t)=\sfL$ on a Borel set
containing $\rmJ_u$ and where $\Cmd u\sfd$ is concentrated.

Setting $a_h(t):=\psi_h^*(\sF_h(t))-\cP_h(t,u_h(t),\sF_h(t))$, by
\eqref{eq:26a} is immediate to see that 
\begin{equation}
  \label{eq:135}
  \liminf_{h\in H}a_h(t)\ge a(t).
\end{equation}
% $H'\subset H$ be an arbitrary subsequence; 
% by a standard diagonal argument, starting from \eqref{eq:56} we can find a suitable subsequence
% $H''\subset H'$ such that 
% By \eqref{eq:62} we obtain
% \begin{equation}
%   \label{eq:74}
%   \limsup_{h\in H''}\cP_h(t,u_h(t))\le \cP(t,u(t))\quad \text{for $\Leb 1$-a.e.\
%   }t\in (0,T).
% \end{equation}
Since \eqref{eq:5bis} and (C1) also yield
\begin{equation}
  \label{eq:76}
  \sup_h\int_0^T \big(\cP_h(t,u_h(t),\sF_h(t))\big)_-\,\d t<\infty,
\end{equation}
Fatou's lemma and \eqref{eq:135} imply
\begin{equation}
  \label{eq:77}
  \int_0^T\big(a(t)\big)_+\,\d t\le 
  \liminf_{h\in
    H}\int_0^T\big(a_h(t)\big)_+\,\d t<\infty,
\end{equation}
and the inequality $a(t)\ge -\cP(t,u(t),\sfL)$ shows that $a\in L^1(0,T)$.
\eqref{eq:24strong} and Fatou's Lemma then yield
\begin{equation}
  \label{eq:75}
  \begin{aligned}
    \liminf_{h\in H}&\int_0^T
    \Big(\psi_h^*(\sF_h(t))-\cP_h(t,u(t),\sF_h(t))\Big)\,\d t\\&\ge
    % \int_0^T a(t)\,\d t=
    \int_0^T \Big(\psi^*(\sF(t))-\cP(t,u(t),\sF(t))\Big)\,\d t.
  \end{aligned}
\end{equation}
Recalling \eqref{eq:72}, we consider now the
measures 
\begin{equation}
  \label{eq:126}
  \begin{aligned}
    \eta_h:=&\Big(\psi_h^*\big(\sF_h\big)
    -\cP_h(\cdot,u_h,\sF_h)\Big)\Leb 1,\\
    \mu_{h}:=&\Big(\psi_h\big(|\dot
    u_h|_{\sfd_h}\big)+\psi_h^*\big(\sF_h\big)-\cP_h(\cdot,u_h,\sF_h)\Big)\Leb
    1= \nu_{h,0}+ \eta_h;
  \end{aligned}
\end{equation}
up to extracting a further subsequence, it is not restrictive to
assume that $\eta_h\weaksto \eta,$ $\mu_h\weaksto\mu\ge \nu_0+\eta$ in the duality with continuous
functions. Fatou's Lemma and the previous arguments easily imply
\begin{equation}
  \label{eq:128}
  \eta\ge \Big(\psi^*\big(\sF\big)-\cP\big(\cdot,u,\sF\big)\Big)\Leb 1.
\end{equation}
Since $\lim_{h\in
  H}\mu_h([0,T])=\mu([0,T])$, combining \eqref{eq:61} and 
\eqref{eq:75} inequality \eqref{eq:5bv} for $t=T$ follows if we show that 
\begin{equation}
  \label{eq:78}
  \begin{aligned}
    \mu([0,T])&\ge \int_{0}^{T} \Big(\psi\big(|\dot
    u|_\sfd\big)+\psi^*(\sF) -\cP(\cdot,u,\sF)\Big) \,\,\d r
    \\&+\sfL\int_0^T \d\Cmd u\sfd+\pJmp{\sfd,\frf}(u;[0,T]).
  \end{aligned}
\end{equation}
Now, \eqref{eq:125} and \eqref{eq:128} imply
that 
\begin{equation}
  \label{eq:127}
  \mu\ge \nu_0+\eta\ge \Big(\psi\big(|\dot
  u|_{\sfd}\big)+\psi^*\big(\sF\big)-\cP(\cdot,u,\sF)\Big)\Leb 1+\sfL\Cmd u\sfd.
\end{equation}
Since the atomic and the diffuse part of a measure are mutually
singular, \eqref{eq:78} ensues if for every $t\in [0,T]$ with
$\mu(\{t\})>0$ we have
\begin{equation}
  \label{eq:82}
  \mu(\{t\})\ge \tricost{\sfd,\frf}t{u(t_-),u(t),u(t_+)}.
\end{equation}
In order to prove \eqref{eq:82} we just take two sequences $r_h^-<t<r_h^+$, $h\in \N$, converging
monotonically to $t$ such that 
\begin{equation}
  \label{eq:129}
  u_h(r_h^-)\topto u(t_-),\quad
  u_h(r_h^+)\topto u(t_+),
\end{equation}
and we set 
\begin{displaymath}
  \sfs_h(r):=r+\int_t^r \Big(\psi_h\big(|\dot
  u_h|_{\sfd_h}(\tau)\big)+\psi^*_h\big(\cF_h(\tau,u_h(\tau))\big)\Big)\,\d\tau,\quad
  \sfs_h^\pm:=\sfs_h(r_h^\pm).
\end{displaymath}
Since  $\cF_h(\tau,u_h(\tau))\le \sF_h(\tau)$ and 
\begin{equation}
  \label{eq:144}
  \limsup_{h\in H}
  \int_{t_{h\,-}}^{t_{h\,+}} \cP_h(t,u_h(t),\sF_h(t))\,\d t=0
\end{equation}
by \eqref{eq:24strong},
taking into account the definition \eqref{eq:126} of $\mu_h$ we obtain
\begin{displaymath}
\limsup_{h\in H''}(\sfs_h^+-\sfs_h^-)\le \limsup_{h\in
  H''}\mu_h([t_h^-,t_h^+])\le  \mu(\{t\}),
\end{displaymath}
and up to
extracting a suitable subsequence we can assume that 
$\sfs_h^\pm\to \sfs^\pm$ as $h\to\infty$. We denote by
$\sfr_h:=\sfs_h^{-1}$ the inverse map of $\sfs_h$: $\sfr_h$ is
$1$-Lipschitz, monotone, and maps $[\sfs_h^-,\sfs_h^+]$ onto
$[r_h^-,r_h^+]$.
We also set
\begin{equation}
  \label{eq:84}
  \vartheta_h(s):=
  \begin{cases}
    u_h(\sfr_h(s))&\text{if }s\in [\sfs_h^-,\sfs_h^+],\\
    u_h(r_h^+)&\text{if }s\ge \sfs_h^+,\\
    u_h(r_h^-)&\text{if }s\le \sfs_h^-,
  \end{cases}
\end{equation}
so that, in particular, we have 
\begin{displaymath}
  \vartheta_h(\sfs_h^\pm)=u_h(r_h^\pm),\qquad
  \vartheta_h(t)=u_h(t).
\end{displaymath}
 We observe that $\tilde\cE_h(\sfr_h(s),\vartheta_h(s))\le C$ and that
 the functions
$\vartheta_h$ are uniformly $\sfd_h$-Lipschitz: to show this fact, we choose $f\in
[0,\sfL)$ in such a way that $\sup_h\psi_h^*(f)\le 1$;
the inequality $\psi_h(v)\ge fv-\psi_h^*(f)$ yields
\begin{displaymath}
  \dot \sfs_h(r)\ge 1+\psi_h(|\dot u_h|_{\sfd_h}(r) )\ge 1+f\,|\dot u_h|_{\sfd_h}(r)-\psi^*_h(f)\ge f\,|\dot u_h|_{\sfd_h}(r)
\end{displaymath}
so that 
\begin{align}
  \notag
  f\sfd_h(\vartheta_h(\alpha),&\vartheta_h(\beta))=
  f\sfd_h(u_h(\sfr_h(\alpha)),u_h(\sfr_h(\beta))\le 
  \int_{\sfr_h(\alpha)}^{\sfr_h(\beta)}f\,|\dot u_h|_{\sfd_h}(r)\,\d
  r\\
  \label{eq:143}
  &\le   \int_{\sfr_h(\alpha)}^{\sfr_h(\beta)}\dot\sfs_h(r)\,\d r\le \sfs_h(\sfr_h(\beta))-\sfs_h(\sfr_h(\alpha))=\beta-\alpha,
\end{align}
whence the uniform $\sfd_h$-Lipschitz continuity of $\vartheta_h$.

By arguing as in the proof of Theorem \ref{thm:compactness}, we can
find a compact interval $I$ containing all the intervals
$[\sfs_h^-,\sfs_h^+]$ and a further subsequence such that $\vartheta_h(s)\to\vartheta(s)$
for every $s\in I$: it follows from \eqref{eq:143} that $\vartheta$ is $f^{-1}$-Lipschitz with
respect to $\sfd$, $\vartheta(\sfs^\pm)=u(t_\pm)$,
$\vartheta(t)=u(t)$, and 
\begin{equation}
  \label{eq:130}
  |\dot\vartheta_h|_{\sfd_h}\weakto^* m\quad\text{in
  }L^\infty(I)\quad\text{with }m\ge |\dot\vartheta|_\sfd.
\end{equation}
Moreover, by using the elementary inequality
\begin{equation}
  \label{eq:85}
  \psi_h(v)+\psi_h^*(f)\ge (f\lor a)v-\psi_h^*(a)\quad\text{for every
  }a\in [ 0,\sfL),
\end{equation}
and recalling that $r_h^+-r_h^-\to0$, 
$\psi_h^*(a)\to\psi^*(a)<\infty$, and \eqref{eq:144}, we have
\begin{align}
  \notag\mu(\{t\})&\ge\limsup_{h\to\infty}\mu_h([r_h^-,r_h^+])
  \\&\topref{eq:144}=\limsup_{h\in H}\int_{r_h^-}^{r_h^+}\Big(\psi_h\big(|\dot
  u_h|_{\sfd_h}(\tau)\big)+\psi^*_h\big(\cF_h(\tau,u_h(\tau))\big)\Big)\,\d\tau\\\notag&
  \topref{eq:85}\ge
  \liminf_{h\in H}\int_{r_h^-}^{r_h^+} \Big((\cF_h(\tau,u_h(\tau))\lor a) |\dot
  u_h|_{\sfd_h}(\tau)-\psi_h^*(a)\Big)\,\d \tau
  \\&\notag
  \ge
  \liminf_{h\in H}\int_{\sfs_h^-}^{\sfs_h^+}
  \Big((\cF_h(\sfr_h(s),\vartheta_h(s))\lor a) |\dot
  \vartheta_h|_{\sfd_h}(s)\Big)\,\d s -(r_h^+-r_h^-)\psi_h^*(a)
    \\\notag
    &\ge
    \liminf_{h\in H}\int_{I}
  \Big((\cF_h(\sfr_h(s),\vartheta_h(s))\lor a) |\dot
  \vartheta_h|_{\sfd_h}(s)\Big)\,\d s 
\\\label{eq:145}
&\ge 
  \int_{I}
  \Big((\cF(t,\vartheta(s))\lor a) m(s)\Big)\,\d s
    \\\notag&
\ge 
  \int_{\sfs^-}^{\sfs^+}
  \Big((\cF(t,\vartheta(s))\lor a) |\dot\vartheta|_\sfd(s)\Big)\,\d s
  \ge \tricost{\sfd,\frf}t{u(t_-),u(t),u(t_+)},
\end{align}
viz.\ the desired \eqref{eq:82}.
For the last $\liminf$ inequality in \eqref{eq:145} we used a result proved in the next lemma.
\end{proof}

\begin{lemma}
  Let $I$ be a bounded interval in $\R$, $F,m,\ F_h,m_h:I\to [0,\infty)$, $h\in \N$, be measurable functions satisfying
  \begin{equation}
    \label{eq:131}
    \liminf_{h\to\infty}F_h(s)\ge F(s)\quad\text{for $\Leb 1$-a.a.\
      $s\in I$},\quad
    m_h\weakto m\quad\text{in }L^1(I).
  \end{equation}
  Then
  \begin{equation}
    \label{eq:132}
    \liminf_{h\to\infty}\int_I F_h(s)\,m_h(s)\,\d s\ge \int_I F(s)\,m(s)\,\d s
  \end{equation}
\end{lemma}
\begin{proof}
  Let us set $G_k(s):=\inf_{h\ge k}F_h(s)\land k$, $k\in \N$. Since 
  $G_k(s)\le F_h(s)$ for every $h\ge k$ and $G_k\in L^\infty(I)$ we
  have for every $k\in \N$
  \begin{displaymath}
    \liminf_{h\to\infty}\int_I F_h(s)\,m_h(s)\,\d s\ge
    \liminf_{h\to\infty}\int_I G_k(s)\,m_h(s)\,\d s=\int_I
    G_k(s)\,m(s)\,\d s.
  \end{displaymath}
  On the other hand, for $\Leb 1$-a.a.\ $s\in I$, $k\mapsto G_k(s)$ is
  a nondecreasing sequence
  converging to $\liminf_{h\to\infty}F_h(s)$, so that by monotone
  convergence
  \begin{displaymath}
    \lim_{k\to\infty}\int_I
    G_k(s)\,m(s)\,\d s\ge \int_I F(s)\,m(s)\,\d s.\eqno{\qedhere}
  \end{displaymath}
\end{proof}
\noindent
In the case of a limit metric dissipation function $\psi$ with
superlinear growth we have a completely analogous result, which can be
compared
with \cite[Theorem 4.8]{Mielke-Rossi-Savare12}. We omit the 
similar proof: it can be carried out without assuming
\eqref{eq:26}, which has been used only to characterize the
contribution of jump part in the energy dissipation inequality \eqref{eq:5bv}.
\begin{theorem}[Convergence in the superlinear case]
  \label{thm:convergence2}
  Let us assume that $\psi$ is superlinear and for $h\in\N$
  let 
  $(\XX,\sfd_h,\cE_h,\cF_h,\cP_h)$ be a family of evolution
  systems satisfying {\upshape (C1,2,3) and (C4$_{E,FP}$)} with respect to a sequence
  $\bar u_h\in \XX$ $\sigma$-converging to $\bar u$ and
  let  $(\XX,\sfd,\cE,\cF,\cP)$ be an upper gradient evolution system.\\
  Let  $u_h\in \AC{}{0,T}{(\XX,\sfd_h)}$, $\sF_h\in \rmB_+([0,T])$, $h\in H$, be sequences of curves
  satisfying the $\psi_h$-energy dissipation inequality
  \eqref{eq:5bis},
  pointwise converging to $u$ in $(\XX,\sigma)$, with $u_h(0)=\bar
  u_h$, and satisfying  \eqref{eq:24strong}.
  % \begin{equation}
  %   \label{eq:24strong}
  %   \lim_{h\in H} \cE_h(0,\bar u_h)=\cE(0,\bar u),\quad
  %   r\mapsto \big[\cP_h(r,u_h(r),\sF_h(r))\big]_+ \text{ are equi-integrable in $(0,T)$}.
  % \end{equation}
  %\begin{description}
  % \item[$(\sfL=\infty):$] If $\psi$ is superlinear, then $u\in \AC{}{0,T}{(\XX,\sfd)}$
  %   and satisfies the energy-dissipation inequality \eqref{eq:5}.
  % \item[$(\sfL<\infty):$] If $\psi$ has $\sfL$-linear growth, 
  Then $u\in \AC{}{0,T}{(\XX,\sfd)}$ is a $\psi$-gradient flow
  of the limit system and 
  \begin{gather}
    \label{eq:69bis}
    \lim_{h\in H}\cE_h(t,u_h(t))=\cE(t,u(t))\quad\forevery t\in [0,T].
  \end{gather}
%\end{description}
\end{theorem}

\subsection*{Examples}
%\label{sec:examples}
\subsection*{$\lambda$-convex energies.}
Let $(\XX_h,\|\cdot\|_h)$, $(\XX,\|\cdot\|)$ be a family of Banach spaces
such that $\XX_0\subset \XX_h\subset \XX$ with continuous and dense inclusions.
Setting $\|u\|_h=\infty$ if $u\in
\XX\setminus \XX_h$ we suppose that $\Gamma(\XX)\text{-}\lim_{h\to\infty}\|\cdot\|_h=\|\cdot\|$.
Let $\cE:\XX_0\to(-\infty,+\infty]$ be a proper, $\lambda$-convex
functional, i.e.\ satisfying
for every $u_0,u_1\in \XX_0$
\begin{displaymath}
  \cE((1-\theta)u_0+\theta u_1)\le
  (1-\theta)\cE(u_0)+\theta\cE(u_1)-\frac\lambda2\theta(1-\theta)\|u_1-u_0\|
  \quad\forall\,\theta\in [0,1].
\end{displaymath}
We consider
a time-dependent functional $\ell\in \rmC^1([0,T];\XX_0')$, we suppose that
$\cE(t,u):=\cE(u)-\langle \ell(t),u\rangle$ has compact sublevels on
$\XX_0$, and we extend it to $\XX$ by setting $\cE(t,u)=\infty$ if $u\in
\XX\setminus \XX_0$.
We set $\cP(t,u)=\partial_t \cE(t,u)=\langle \ell'(t),u\rangle$ and 
$$\cF_h(t,u):=\min\big\{\|\xi-\ell(t)\|_{h}^*:\xi\in \partial_h
\cE(u)\big\},$$
where $\partial_h$ is the Frech\'et subdifferential of the restriction
of $\cE$ in $\XX_h$. Notice that $\partial_h\cE\subset \XX_0'$ and 
it is not difficult to check (see 
\cite{Rossi-Mielke-Savare08,Mielke-Rossi-Savare12})
that $(\XX,\sfd_h,\cE,\cF_h,\cP)$ and $(\XX,\sfd,\cE,\cF,\cP)$
are upper gradient evolution systems and all the
assumptions (C1,2,3,4) are satisfied.

\subsection*{Dirichlet energy and double-well potentials.}
Here is a concrete example of the above setting.
Consider a bounded open set $\Omega\subset \Rd$, a function $W\in
\rmC^2(\R)$ with $\inf_\R W>-\infty,$ $\inf_\R W''>-\infty$, and $\ell\in
\rmC^1([0,T];L^2(\Omega))$. In the space $\XX:=L^2(\Omega)$ endowed
with
the strong $L^2$-topology we set 
\begin{align*}
  \cE(t,u):=&\int_\Omega \Big(\frac 12|\nabla
  u|^2+W(u)-\ell(t)\,u\Big)\,\d x \quad\text{if }u\in H^1_0(\Omega),\
  W\circ u\in L^1(\Omega),\\
  \cE(t,u):=&+\infty\quad\text{otherwise}.
\end{align*}
We choose a sequence of exponents $p_h>1$ converging to $1$ as
$h\to\infty$, and initial data
$\bar u_h\in H^1_0(\Omega)$ with $\cE(0,\bar u_h)<\infty$ strongly
converging to $\bar u$ in $H^1_0(\Omega)$ with $W\circ u_h\to W\circ
u$ in $L^1(\Omega)$. 

We let $\sfd_h$ be the distance induced by
the $L^{p_h}(\Omega)$ norm, $\sfd$ the $L^1(\Omega)$-distance,
$\psi_h(v):=\frac 1{p_h}v^{p_h}$ , and
\begin{equation}
  \label{eq:134}
  \cF_h(t,u):=\|-\Delta u+W'(u)-\ell(t)\|_{L^{p_h^*}(\Omega)},\quad
  \cP(t,u)=\int_\Omega \ell'(t)\, u\,\d x,
\end{equation}
with $\cF_h(t,u)=+\infty$ if $-\Delta
u+W'(u)-\ell(t)\not\in L^{p_h}(\Omega)$; $\cF$ has the analogous expression in
$L^\infty(\Omega)$.

Applying the results of \cite[\S 7.2]{Rossi-Mielke-Savare08} we see
that $(\XX,\sfd_h,\cE,\cF_h,\cP)$ and $(\XX,\sfd,\cE,\cF,\cP)$ are
upper gradient evolution systems and 
for every $h\in \N$ there exists a solution $u_h$ of the
$\psi_h$-gradient flow. It is also easy to check that all the
assumptions (C1,2,3,4) are satisfied
so that, up to subsequences, $u_h(t,\cdot)$ converge to $u(t,\cdot)$ 
in $L^2(\Omega)$ at every time $t$ with convergence of the energies
$\cE(t,\cdot)$
and $u$ is a BV solution of the rate-independent evolution 
governed by $(\XX,\sfd,\cE,\cF,\cP)$.

\subsection*{Marginal functionals.}
In the finite dimensional setting described in Section
\ref{sec:smooth} 
(here
 we also assume that the norms $\|\cdot\|_u$ are independent of $h$),
let us consider the marginal functional \eqref{eq:27} and the 
couple $(\cF,\cP)$ given by \eqref{eq:137} and \eqref{eq:152}.

It is not difficult to see that 
$(\XX,\sfd,\cE,\cF,\cP)$ is an upper gradient evolution system.
$\psi_h$-gradient flows in the superlinear case can be obtained by
applying the results of \cite{Mielke-Rossi-Savare12}: they in
particular solve the
differential inclusions \eqref{eq:86marginal}.
Existence of a BV solution and convergence of the $\psi_h$ gradient
flows can thus be obtained by applying Theorems \ref{thm:compactness}
and \ref{thm:convergence}.

\subsection*{Acknowledgments.}

The authors thank the anonymous reviewer for the careful 
reading of the manuscript.

\renewcommand{\sc}{}
\def\cprime{$'$}

% \bibliographystyle{siam}
% \bibliography{bibliografia2012}
\end{document}